\documentclass[a4paper,12pt]{amsart} %{article}
\title[Double ramification cycles]{The Hodge-Double-Ramification conjecture and Mumford's formula on the universal Picard stack}

\usepackage{amsmath, amssymb,amsthm,, tikz}% amsfonts, url,, tikz, wrapfig, newclude, }
\usepackage{mathtools} % for \coloneqq (and presumably other stuff!)
\usepackage{hyperref}
\usepackage{xcolor}
\hypersetup{
colorlinks,
linkcolor={black},
citecolor={blue!50!black},
urlcolor={blue!80!black}
} 

\usepackage{verbatim}% for multiline comments

\usepackage[bbgreekl]{mathbbol}
\DeclareSymbolFontAlphabet{\mathbbm}{bbold}
\DeclareSymbolFontAlphabet{\mathbb}{AMSb}

\usepackage{tabularx}
\usepackage{longtable}
\numberwithin{equation}{subsection}

\newcommand{\mat}[1]{\begin{bmatrix}#1\end{bmatrix}}

\usepackage{enumitem}

\usepackage{cleveref}
\crefformat{equation}{(#2#1#3)}
\let\oref\ref
\AtBeginDocument{\renewcommand{\ref}[1]{\Cref{#1}}}

\renewcommand{\tilde}[1]{\widetilde{#1}}

\usepackage{pgf,tikz}
\usetikzlibrary{matrix, calc, arrows}
\usepackage{tikz-cd} %this is the old version, but it also the one the arXiv supports at the moment (2016)

\newcommand{\on}[1]{\operatorname{#1}}
\newcommand{\bb}[1]{{\mathbb{#1}}}

\newcommand{\ca}[1]{{\mathcal{#1}}}
\newcommand{\bd}[1]{{\mathbf{#1}}}
\newcommand{\ul}[1]{{\underline{#1}}}

\newcommand{\abs}[1]{\lvert#1\rvert}

\newcommand{\iso}{\stackrel{\sim}{\longrightarrow}}

\theoremstyle{definition}
\newtheorem{definition}{Definition}[section]

\theoremstyle{plain}% default

\newtheorem{proposition}[definition]{Proposition}
\newtheorem{lemma}[definition]{Lemma}
\newtheorem{theorem}[definition]{Theorem}
\newtheorem{corollary}[definition]{Corollary}

\newtheorem{introtheorem}{Theorem}[section]

\newtheorem{preintrotheorem}{Main Theorem}%[section]

\theoremstyle{remark}

\newtheorem{remark}[definition]{Remark}

\newtheorem{notation}[definition]{Notation}

\newcommand{\mdeg}{\ul{\mathrm{de}}\mathrm{g}}

\usepackage{letltxmacro}
\LetLtxMacro{\phiorig}{\phi}
\renewcommand{\phi}{\varphi}

\date{\today}

\author{Alessandro Chiodo}

\address{Sorbonne Université and Université Paris Cité, CNRS, IMJ-PRG, F-75005 Paris, France}

\email{alessandro.chiodo@imj-prg.fr}

\author{David Holmes}
\address{Mathematisch Instituut\\ 
Universiteit Leiden\\
Postbus 9512\\
2300 RA Leiden\\
Netherlands}

\email{holmesdst@math.leidenuniv.nl}

\tikzset{
  symbol/.style={
    draw=none,
    every to/.append style={
      edge node={node [sloped, allow upside down, auto=false]{$#1$}}}
  }
}

\newcommand{\Mbar}{\overline{\ca M}}
\newcommand{\DR}{\mathsf{DR}}

\newcommand{\cat}[1]{\bd{#1}}
\newcommand{\M}{\mathsf{M}}
\newcommand{\ghost}[1][M]{{\overline{\mathsf{#1}}}}
\renewcommand{\log}{\sf{log}}

\newcommand{\gp}{\mathsf{gp}}
\newcommand{\et}{\mathsf{\acute{e}t}}
\renewcommand{\sf}[1]{\mathsf{#1}}

\newcommand{\Chow}{\on{CH}}
\newcommand{\Spec}{\on{Spec}}
\newcommand{\PIC}{\mathfrak{Pic}}
\newcommand{\Sing}{\on{Sing}}
\newcommand{\SSing}{{\mathbb{S}\!\on{ing}}}

\newcommand{\Jac}{\mathsf{Jac}}
\newcommand{\LogPic}{\mathsf{LogPic}}

\newcommand{\logDR}{\sf{LogDR}}
\newcommand{\logDRL}{\sf{LogDRL}}

\newcommand{\op}{\sf{op}}

\newcommand{\vir}{\sf{vir}}

\newcommand{\pcc}{\Omega}%push of chern class
\newcommand{\rpcc}{\Xi}%scaled push of chern class PASTED IN - MAYBE NOT WHAT WE WANT!!
\newcommand{\rgraphs}{\mathsf{G}}

\newcommand{\irr}{\textrm{irr}}

\newcommand{\polarbdl}{\ca H}
\newcommand{\polarcls}{\eta}

\begin{document}
\begin{abstract}The double ramification (DR) cycle associated to a line bundle on a family of curves detects where the line bundle becomes fibrewise-trivial. 
The Hodge-DR Conjecture proposes a formula for powers of the first Chern class of a natural line bundle on the DR cycle, with a number of applications in the computation of Euler characteristics of strata of differentials. In this paper we prove the conjecture, as well as an analogue for the logarithmic DR cycle. The proof of the former proceeds via reduction to a localisation computation of Fan, Wu and You; the proof of the latter is based on the Thom--Porteous formula, and as a special case gives a shorter proof of a recent result of Holmes, Molcho, Pandharipade, Pixton and Schmitt. Along the way we develop an analogue of Mumford's formula for the Chern character of the universal line bundle on the universal jacobian over the moduli space of twisted curves, generalising work of Mumford, Chiodo, and Pagani--Ricolfi--van Zelm. 
\end{abstract}
\maketitle
%
%Main things to do: 
%\begin{enumerate}
%%\item polish intro
%\item fix broken references
%%\item do we mention invariances somewhere in the intro? No; todo. 
%%\item make clear structure of proof of log Pixton. How do we apply Theorem D? Do we want to actually state the theorem? 
%%\item add new results (PRvZ, invariances) to the "defs and statements of main results" section. 
%%\item Statement of hodge-DR is not in intro? At least it does not have its own subsection. 
%\end{enumerate}
%
%Narrative: 
%
%What are our main results: 
%
%- Proof of hodge DR conjecture via localisation; for $u=0$ it's just the same as BHPSS proof. 
%
%- proof of log Hodge DR conjecture via Thom--Porteous and GRR; new proof even for $u=0$. 
%
%What are smaller results still to point out: 
%
%- Mumford's formula on Picard stack
%
%- generalise Nicola et al's computation
%
%- invariances

%\shorttableofcontents{Contents}{1}

\tableofcontents

\section*{Introduction}

A fundamental problem in the study of algebraic curves is to determine when a given divisor can be realised as the zeros and poles of a rational function. An algorithm for a single curve was given by Abel and Jacobi in the 19th century. A more modern approach (with numerous applications in enumerative geometry and integrable systems, including \cite{BSSZ,Cav,clader2018powers,clader2018pixton,Schmitt2016Dimension-theor,BGR,Bae_taut,CSS,cavalieri2022pluricanonical,Ranganathan2022Logarithmic-Gro,BR21,BLS24,Blot2024Rooted-trees-wi,cavalieri2024k}) is to study this problem over the moduli space $\Mbar_{g,n}$ of stable $n$-marked curves of genus $g$. We begin this introduction by discussing what happens on the open locus $\ca M_{g, n}^\irr\subseteq \Mbar_{g,n} $ of irreducible curves. Given a vector of integers $\ul a = (a_1, \dots a_n) \in \bb Z^n$ summing to $0$, we form the divisor $\sum_i a_i [p_i]$ on the universal curve $\pi\colon C \to \ca M^\irr_{g,n}$, where the $p_i$ are the markings. The \emph{double ramification} locus
\begin{equation*}
\on{DRL}(\ul a) = \big\{(C;x_1,\dots,x_n)\in \ca M^\irr_{g,n}:
 \ca O({\textstyle\sum_i a_i[p_i]})\cong \ca O\big\}\stackrel{i}{\hookrightarrow} \ca M^\irr_{g,n}
\end{equation*}
measures exactly where the divisor $\sum_i a_i [p_i]$ is principal; it is so named because it can be seen as the locus of curves admitting a map to $\bb P^1$ with ramification over $0$ and $\infty$ prescribed by the positive and negative $a_i$. The double ramification locus is closed and of \emph{expected} codimension $g$; more precisely, we can construct a natural \emph{virtual fundamental class} in $\Chow^g(\ca M^\irr_{g,n})$. To define this class, we let $\Jac_{g,n}^\irr \to \ca M_{g, n}^\irr$ be the Jacobian of the universal curve $C$; this is the moduli space of degree-$0$ line bundles on $C$, and is a semiabelian scheme\footnote{Each fibre is an extension of an abelian variety by a torus. } over $\ca M_{g, n}^\irr$ of relative dimension $g$. We write $e\colon \ca M_{g, n}^\irr \to \Jac_{g, n}^\irr$ for the zero section of the Jacobian (represented by the trivial bundle on $C$), and $\sigma_\ul a\colon \ca M_{g, n}^\irr \to \Jac_{g, n}^\irr$ for the section representing the line bundle $\ca O(\sum_i a_i [p_i])$. Then 
\begin{equation}
\on{DRL}(\ul a) = \sigma_\ul a^{-1} e = e^{-1}\sigma_\ul a \stackrel{i}{\hookrightarrow} \ca M^\irr_{g,n}, 
\end{equation}
and we define the virtual fundamental class, the \emph{double ramification cycle}, as the Gysin pullback 
\begin{equation}
\DR(\ul a) = \sigma_\ul a^! [e] = e^![\sigma_\ul a] \in \Chow^g(\ca M^\irr_{g,n}). 
\end{equation}
Furthermore, the base change of the universal curve 
via the above map $i$ yields a 
family of curves $\pi_{\DR}\colon C_{\DR}\to \on{DRL}(\underline a)$
equipped with a fibrewise-trivial line bundle 
$\ca L_{\DR}=\ca O({\textstyle\sum_i a_i[p_i]})$; therefore,
via direct image we obtain the \emph{tautological}\footnote{The dual is because this was originally constructed in \cite{Chen2022A-tale-of-two-m} via a projective embedding, making the opposite sign choice more natural. } line bundle $\polarbdl=(\pi_{\DR})_*(\ca L_{\DR})^\vee$ with first Chern class $\polarcls$ on $\on{DRL}(\underline a)$, and for each positive integer $u$ a new 
 codimension-$(g+u)$ cycle
\begin{equation}\label{eq:groundlb}
\DR(\ul a)\cdot \polarcls^u. 
\end{equation}
Observe that we can replace $\ca O(\sum_i a_i [p_i])$ in the above constructions by \emph{any} line bundle $\ca L$ on $C$ which is fibrewise of degree 0, and the same discussion goes through; from now on we work in this increased generality. The case $\ca  L = \omega_{C/\ca M_{g, n}^\irr}^{\otimes k}(\sum_i a_i [p_i])$ for coefficients $a_i$ summing to $k(2g-2)$ makes important connections to spaces of differentials, see \cite{Farkas2016The-moduli-spac,Schmitt2016Dimension-theor,Bainbridge2019The-moduli-spac,Chen2022A-tale-of-two-m}, in particular the integral of powers of the first Chern class $\polarcls$ of the tautological line bundle $\polarbdl$ there is important for the computation of Euler characteristics of strata of differentials, see \cite[Theorem 1.3]{costantini2022chern}. 

On $\ca M_{g, n}^\irr$, and assuming $n$  positive, it is 
easy to express the class of the cycle $\DR(\ul a)\cdot \polarcls^u$ as 
a degeneracy class of a map of vector bundles. We can consider the exact sequence $0\to \ca L(-D)\to \ca L\to \ca L\mid_D\to0$ for a divisor $D$
%
%\ale{In the following sentence it is not clear what we really want to say. I think  ``We can consider the exact sequence $0\to \ca L(-D)\to \ca L\to \ca L\mid_D\to$ for a suitable divisor $D$ \dots'' is enough} We can consider 
%the complex $\ca L|_D \to R^1\pi_*\ca L(-D)$ representing $\ca L$ 
%for a suitable divisor $D$ 
supported on the markings and we choose $D$ sufficiently positive
so that
%the terms of the 
%complex are 
%acyclic for $R\pi_*$ and the complex $R^\bullet \pi_*\ca L$ is 
%quasiisomorphic to 
\begin{equation}\label{eq:irr_complex}
R^\bullet \pi_*\ca L=[\pi_*\ca L|_D \xrightarrow{\phi} R^1\pi_*(\ca L(-D))] 
\end{equation}
in the derived category.
Then the degeneracy locus of $\phi$ is exactly the 
locus where $\ca L$ admits a non-zero global section; {i.e.} 
the locus $\on{DRL}(\ca L)$. 
If we perform the same computation, not with the line bundle $\ca L$ on $C/\ca M_{g, n}^\irr$, but rather with the universal line bundle on $\Jac^\irr_{g, n}$, then this degeneracy locus is exactly the zero section of $\Jac^\irr_{g, n}$, which has the expected codimension and has fundamental class equal (by the Thom--Porteous formula) to $c_g(-R\pi_*\ca L)$ where $R\pi_*\ca L=\sum_{i} (-1)^i R^i\pi_*\ca L$. Because 
$\DR(\ca L)$ is defined as the pullback of $[e]$ from $\Jac^\irr_{g, n}$, and  
forming Chern classes commutes with pullback, we deduce 
$$\DR(\ca L) = c_g(-R\pi_* \ca L).$$ 
The intersections with powers of the first Chern class of the tautological line bundle $\polarbdl$ can be computed in a similar way, by realising that the degeneracy locus of \ref{eq:irr_complex} is really supported on the projectivisation $\bb P(\pi_*\ca L(D))$, and the line bundle $\polarbdl$ is simply $\ca O(1)$ on this projective bundle. By a slightly more careful application of the Thom--Porteous formula, we obtain the key relation 
\begin{equation}\label{eq:irr_formula}
\DR(\ca L)\cdot \polarcls^u = c_{g+u}(-R\pi_* \ca L). 
\end{equation}
Then, via Grothendieck Riemann--Roch, the latter can be expressed in terms of tautological classes on $\ca M_{g, n}^\irr$.

\subsection{Extending the DR cycle to stable curves}
For most of the applications mentioned above it is necessary to extend the constructions and computations from irreducible to stable curves. The definition of the double ramification cycle via the Jacobian carries over with some modifications. The Jacobian $\Jac_{g, n}$ (moduli space of line bundles of \emph{total} degree 0) of the universal curve $\pi\colon C \to \Mbar_{g,n}$ is no longer separated, so its unit section $e$ is no longer closed, and pulling it back along the section 
\begin{equation}
\sigma_\ca L \colon \Mbar_{g,n} \to \Jac_{g, n}
\end{equation}
classifying $\ca L$ does not yield a closed subspace of $\Mbar_{g,n}$. However, we can  rectify this problem by replacing the unit section $e$ by its Zariski closure $\bar e$ inside $\Jac_{g, n}$, and defining 
\begin{equation}
\on{DRL}(\ca L) = \sigma_\ca L^{-1}\bar e \hookrightarrow \Mbar_{g, n}
\end{equation}
and
\begin{equation}\label{eq:DR_def_intro}
\DR(\ca L) = \sigma_\ca L^![\bar e] \in \Chow^g(\Mbar_{g, n}). 
\end{equation}
The tautological line bundle $\polarbdl$ on $\on{DRL}(\ca L)$ can be extended similarly, and we still write $\polarcls$ for its first Chern class, see \ref{sec:computing_DR}. 
Note, however, that pulling back does \emph{not} commute 
with Zariski closure, and indeed the double ramification locus on $\Mbar_{g, n}$ is rarely equal to 
the closure of the double ramification locus on $\ca M_{g, n}^\irr$. 

The closure $\bar e$ is in general highly singular, and does not admit a natural functorial description. However, there is a natural resolution of singularities $\tilde e \to \bar e$ coming from log geometry (see \ref{sec:computing_DR}), and the map $\sigma_{\ca L}^{-1}\tilde e \to \Mbar_{g,n}$ naturally factors as a closed immersion to a log blowup\footnote{Roughly, an iterated blowup in boundary strata. } $\Mbar_{g,n}^\ca L \to \ca M_{g,n}$. We define the log double ramification cycle 
\begin{equation}\label{eq:logDR_def_intro}
\logDR(\ca L) = \sigma_\ca L^![\tilde e] \in \Chow^g(\Mbar_{g, n}^\ca L), 
\end{equation}
a cycle on $\Mbar_{g,n}^\ca L$ which pushes forward to $\DR(\ca L)$, and which has a number of natural properties which $\DR(\ca L)$ itself lacks, see \cite{Holmes2017Multiplicativit}. Again, the tautological line bundle $\polarbdl$ can be extended, and we again write $\polarcls$ for its first Chern class, see \ref{sec:computing_DR}. 

\subsection{Computing the DR and logDR cycles}

The class $c_{g}(-R\pi_* \ca L)$ still makes sense on $\Mbar_{g, n}$ and on $\Mbar_{g,n}^\ca L$, and still computes a degeneracy class representing the locus where $\ca L$ admits a non-zero global section, but it does \emph{not} equal $\DR(\ca L)$ or $\logDR(\ca L)$; admitting a non-zero global section does \emph{not} imply that the line bundle is trivial, since a non-zero global section can still vanish on several irreducible components. 

We know of two ways to correct this, one of which (taking roots of $\ca L$) yields a formula for $\DR(\ca L)$, the other (replacing $\ca L$ by a quasi-stable representative) yields a formula for $\logDR(\ca L)$; we discuss these in turn.

\subsubsection{Extracting \texorpdfstring{$r$}{r}th roots and setting \texorpdfstring{$r=0$}{r=0}} 
Over the past decade, Janda, Pandharipande, Pixton, and Zvonkine (JPPZ)
developed a technique 
\cite{ppz3spin, {Janda2016Double-ramifica}, Janda2018Double-ramifica}
identifying 
several instances of virtual classes 
within polynomial expressions in $H^*(\Mbar_{g,n})$ (and Chow rings)
via the moduli of $r$th roots
$$\epsilon\colon \Mbar_{g,n}^{\ca L,1/r} \to \Mbar_{g, n}.$$ 
In the present case, this can be made precise in relatively simple terms.
The above map $\epsilon$ is the finite cover 
classifying roots $\ca L^{1/r}$
of the given line bundle $\ca L$ on the universal stable curve,
with two small caveats: $(i)$ 
the curve where we extract the root 
is an enriched stack-theoretic version of the universal stable curve, with stabilisers of order $r$
at all nodes (an \emph{$r$-twisted} curve); $(ii)$ each point has an extra automorphism of order $r$ given by the multiplication by $r$th roots of unity 
on the fibres of $\ca L^{1/r}$.   

Then the main result of \cite{Janda2016Double-ramifica}, ``Theorem 1'', 
can be  simply phrased as
\begin{equation}\label{eq:JPPZintro}
\DR(\ca L)=[r\epsilon_* c_{g}(-R\pi_* \ca L^{1/r})]_{r=0}
\end{equation}
in the case $\ca L = \ca O(\sum_i a_i [p_i])$.
The expression 
between brackets is \emph{eventually} polynomial in $r$ by \cite[Prop.~3$''$]{Janda2016Double-ramifica}, and $[\ \ ]_{r=0}$ indicates that we take the constant term of this polynomial. 
The factor $r$ before the $\epsilon_*$ comes from $(ii)$; it could be avoided 
by rigidifying systematically the extra automorphism. 

% \ale{Minor problem: I have a doubt, shouldn't we avoid some factors not prime to some aut orders
% or is it us and not them? they say r sufficiently large in their paper and we also do}\david{I htink the polynomiality is true v generally, and then it doesn't matter. }

The proof of the above statement was given via {virtual localisation} (as introduced in \cite{graber1999localization}),  relying on the fact that, for $\ca L = \ca O(\sum_i a_i [p_i])$, the cycle $\DR(\ca L)$ can be naturally identified with the push-forward of the virtual fundamental class of rubber maps \cite{Graber2005Relative-virtua,li2001stable,Janda2016Double-ramifica}. 
This lead to several generalisations 
\cite{Janda2016Double-ramifica,Janda2018Double-ramifica, 
Bae2020Pixtons-formula} (the latter relying on a comparison of log and rubber
virtual classes, see \cite{HMPW_comparison} for a foundational survey) and a conjecture involving the tautological line bundle $\polarbdl$ in
\cite[\S6. The Hodge DR Conjecture]{Chen2022A-tale-of-two-m}, 
proven there for $\ca L = \ca O(\sum_i a_i [p_i])$. Our first main result is a proof of the general case of this conjecture (see \ref{sec:hodge_DR} for details). 

%\david{I think it's back to being 1 of at least 2 main theorems. The others being the log version, and the Mumford formula. Probably not the Pixton formula that comes out (Theorem D), that's rather easy/standard. Also, the statement below is kind of complicated; For example we could state it here for Mbar, then mention in a remark that we have a universal version, and refer later? 
%\\
%We should make clear that this is the Hodge-DR conjecture (and also the log version later). Generally need to think through the main theorems once more. }
%\noindent \textbf{Theorem A.} 
\begin{preintrotheorem}\label{preintrotheorem:DR}
{Let $S$ be an algebraic stack, $\pi\colon C \to S$ a family of prestable curves of genus $g$, and $\ca L$ a line bundle on $C$ fibrewise of degree 0. Let $\epsilon\colon S^{1/r} \to S$ be the finite  map parametrising $r$th roots of $\ca L$, with $\ca L^{1/r}$ the universal $r$th root. Let $\polarbdl$ be the tautological line bundle, with first Chern class $\polarcls$, and fix a non-negative integer $u$. Then the expression
$\epsilon_*c_{g+u}(-R\pi_*\ca L^{1/r})$ is eventually a Laurent polynomial in $r$, and we have 
\begin{equation}
\DR(\ca L) \cdot \polarcls^u =[r^{u+1}\epsilon_*c_{g+u}(-R\pi_*\ca L^{1/r})]_{r=0},
\end{equation}
where $[\ \ ]_{r=0}$ stands for the constant term in $r$.}
\end{preintrotheorem}

Our proof makes use of a number of invariance properties for the DR cycle and for Chern characters, which may be of some independent interest; see \ref{sec:invariances}. 

\subsection{Quasi-stable twists of \texorpdfstring{$\ca L$}{L}}

As explained above, the formula $c_{g}(-R\pi_* \ca L)$ fails to correctly detect where $\ca L$ is trivial since there may be non-zero global sections that fail to trivialise $\ca L$; this can only happen when $\ca L$ has non-zero degree on some irreducible components of the curve. However, after pulling back to a suitable blowup of $\tilde C \coloneqq C \times_{\Mbar_{g,n}} \ca M^{\ca L}_{g,n}$, we can twist $\ca L$ by a vertical boundary divisor to obtain a new line bundle with the following properties: 
\begin{enumerate}
\item The cycle $\logDR(\ca L)$ is exactly the virtual fundamental class of the locus where $\tilde{\ca L}$ is fibrewise trivial;
\item The degrees of $\tilde{\ca L}$ on irreducible components are close enough to 0 that $\tilde{\ca L}$ is trivial if and only if it admits a non-zero global section. In particular, $c_{g}(-R\pi_* \tilde{\ca L})$ computes the virtual fundamental class of the locus where $\tilde{\ca L}$ is fibrewise trivial. 
\end{enumerate}

Putting these together yields our second main theorem. It turns out to be harmless to also take roots of $\ca L$, which together with GRR will end up yielding a more satisfying formula in the case $u=0$, giving a new proof and generalisation of the main result of \cite{Holmes2022Logarithmic-double}.

\begin{preintrotheorem}\label{preintrotheorem:logDR}
%\david{Check hypotheses on $C/S$? }
Let $S$ be an algebraic stack, $\pi\colon C \to S$ a family of prestable curves of genus $g$, and $\ca L$ a line bundle on $C$ fibrewise of degree 0. Let $\epsilon\colon S^{1/r} \to S$ be the finite  map parametrising $r$th roots of $\ca L$, with $\ca L^{1/r}$ the universal $r$th root. Let $\polarbdl$ be the tautological line bundle, with first Chern class $\polarcls$, and fix a non-negative integer $u$. Then for \emph{every} integer $r \ge 1$ we have 
%$\epsilon_*c_{g+u}(-R\pi_*\tilde{\ca L}^{1/r})$ is eventually polynomial\,\footnote{As in \ref{eq:JPPZintro}.} in $r$, and we have 
\begin{equation}
\logDR(\ca L) \cdot \polarcls^u =r^{u+1}\epsilon_*c_{g+u}(-R\pi_*\ca L^{1/r}). 
\end{equation}
\end{preintrotheorem}

\subsection{The formulae of Mumford and Pixton}

We have described two ways to express double ramification cycles in terms of Chern classes; in this section we give methods to compute these Chern classes. 

In 1983, Mumford \cite{Mumford1983Towards-an-enum} gave an expression for the Chern character of the derived pushforward of $\omega$ to $\Mbar_{g,n}$ in terms of standard tautological classes: 
\begin{equation}
\on{ch}_m(\pi_!\omega) = \frac{B_{m+1}(1)}{(m+1)!}\left(\kappa_m - \sum_{i=1}^n  \psi_i^m  + \frac{1}{2} j_*\left(\frac{\psi^m - (- \psi')^m}{\psi + \psi'}\right)\right). 
\end{equation}
To explain this formula, we introduce a little notation: $p_1, \dots, p_n$ are the given sections of $\pi$, and $\psi_i = c_1(p_i^*\omega)$, $\kappa_m = \pi_*(c_1(\omega(\sum_i p_i))^{m+1})$. We write $\Sing$ for the locus where $\pi$ is not smooth, with double cover $\SSing \to \Sing$ separating the two branches; there are natural finite maps $i\colon \SSing \to C$ and $j = \pi\circ i\colon \SSing \to S$. The space $\SSing$ carries two natural line bundles, the first given by the cotangent line to the given branch, and the second being the cotangent line to the other branch; we write $\psi$ and $\psi'$ for the respective first Chern classes. Finally, $B_{m+1}$ is the degree-$(m+1)$ Bernoulli polynomial. 

In 2008, the first-named author \cite{Chiodo2008Towards-an-enum}  a  generalised that formula in two directions: first, by allowing $\omega$ to be raised to some power and to be twisted by some multiples of the markings, and second, by considering also roots of the resulting line bundle. To state the result, we fix integers $s$ and $a_1, \dots, a_n$, and define a line bundle 
\begin{equation}\label{eq:twisted_omega}
\ca L = \omega_\log^{\otimes s}\left(-\textstyle \sum_i a_i p_i\right)
\end{equation}
on the universal curve $C$ over $\Mbar_{g,n}$ (here $\omega_\log = \omega(\sum_i p_i))$. Given a positive integer $r$ dividing the relative degree $s(2g - 2 + n) - \sum_i a_i$ of $\ca L$, a compactification $\Mbar{}_{g,n}^{\ca L/r}$ of the space of $r$th roots of $\ca L$ was constructed in \cite{Jarvis2000Compactificatio}, \cite{abramovich2003moduli}, \cite{caporaso2007moduli}, and \cite{Chiodo2008Stable-twisted-}; it carries a universal twisted curve $\pi^r\colon \ca C\to \Mbar{}_{g,n}^{\ca L/r}$, and $\ca C$ carries a universal $r$th root $\ca L^{1/r}$ of $\ca L$. 
Then 
\begin{multline}
\on{ch}_m(\pi^r_!\ca L^{1/r}) = \frac{B_{m+1}(s/r)}{(m+1)!}\kappa_m - \sum_{i=1}^n \frac{B_{m+1}(m_i/r)}{(m+1)!}\psi_i^m  \\+ \frac{1}{2} \sum_{a=0}^{r-1} \frac{r B_{m+1}(a/r)}{(m+1)!}(j_a)_*\left( \frac{\psi^{m} - (-\psi')^m}{\psi + \psi'} \right). 
\end{multline}
Here the classes $\kappa_m$ and $\psi_i$ are constructed by direct analogy with the classical case. The singular locus $\Sing$ of $\pi^r$ consists of nodal singularities with stabilizers of order $r$; the double cover $\SSing \to \Sing$ separating the branches
   naturally decomposes into 
 $r$ substacks where the $\pmb \mu_r$-action at the nodes, which lifts to 
 $\ca L^{1/r}$, yields the character $a\in \bb Z/r\bb Z$:
 $$\bigsqcup_{a=0}^{r-1} i_a\colon \bigsqcup_{a=0}^{r-1}\SSing_a=\SSing\longrightarrow  \Sing\hookrightarrow \ca C.$$
We write $j_a=\pi^r\circ i_a$, and $\widetilde\psi$ and $\widetilde\psi'$ for the 
 Chern classes of the line bundles cotangent to the given branch and to the other branch, respectively; these equal $\frac1r$ times the Chern classes $\psi$ and $\psi'$ of the line bundles cotangent to the coarse branches. %\Dcomment{Did I get the right version of the $\psi$ classes in the formula above? Maybe we don't need the tilde versions here, or maybe we only need them... Also, do we want $q$ or $a$ as the index? Does not matter, but consistency is good.}\Acomment{All is good. I have
 % put a in the formula since I'm using $a$ everywhere}

\subsubsection{Mumford's formula}

We generalise this in two further directions. First, rather than working over the moduli stack of stable marked curves, we allow an arbitrary family of prestable curves, possibly with markings. In particular, our results apply to the universal curve over the stack of stable maps to any target variety. 

Second, rather than considering only line bundles built from powers of $\omega$ and sections, we allow arbitrary line bundles. 
Even over $\Mbar_{g,n}$ this allows significant new room for manoeuvre, as the line bundle $\ca L$ of \ref{eq:twisted_omega} can be twisted by vertical divisors living over the boundary (for $r=1$, this case was treated by Pagani, Ricolfi and van Zelm \cite{Pagani2020Pullbacks-of-un}). 

In order to give a uniform statement, we work in a rather universal case. Both sides of our formula commute with arbitrary pullback, so all other cases can be obtained from this. We write $\frak M_{g}$ for the stack of prestable curves of genus $g$, and for an integer $d$ we write $\PIC_g^d$ for the stack whose objects are pairs $(C/S, \ca L)$ where $C/S$ is a prestable curve of genus $g$, and $\ca L$ is a line bundle on $C$ of relative degree $d$. In place of $\Mbar{}_{g,n}^{\ca L/r}$ we work with the stack $\PIC_g^{1/r,d}$ of twisted $r$th roots of the universal bundle over $\PIC_g^d$, carrying a universal twisted curve $\pi\colon \ca C \to \PIC_g^{1/r, d}$ and universal $r$th root $\ca L^{1/r}$. % (see \ref{def:PIC_roots}). 
\begin{preintrotheorem}\label{thm:intro_ch_formula} 
 We have 
  \begin{multline}\label{ch_formula_correction} 
  \on{ch}_{m}(\pi_!\ca L^{1/r})=
\pi_* \frac{\mathcal B_{m+1}\left(\frac{c_1\ca L}{r}, c_1\omega\right)}{(m+1)!} + \\
\frac{r}{2} \sum_{\substack{{p+q=m+1}\\{p\ge 2}\\0\le a<r}}
 \frac{B_{p}(\frac{a}{r})}{p!q!r^q}(j_{a})_{*}\left((i_a)^*(c_1\ca L)^{q} \frac{\psi^{p-1} - (-\psi')^{p-1}}{\psi + \psi}\right). 
\end{multline}
\end{preintrotheorem}
Here $\mathcal B_m(x,y) = y^mB_m(x/y)\in \bb Q[x,y]$ is the homogenised version of the usual Bernoulli polynomial. Our formula is written in the Chow Cohomology of the Artin stack $\PIC_g^{1/r, d}$. This stack is neither Deligne--Mumford, nor quasi-compact, nor separated, yet its Chow cohomology (with test objects being finite-type separated schemes) is a reasonably well-behaved ring, as explained in \cite[\S 2]{Bae2020Pixtons-formula}. 
A detailed explanation of the terms in this formula can be found in \ref{sec:intro_RR}. 

%\begin{remark}
%    The expression in \ref{ch_formula_correction}  can be re-written in a more uniform way as 
%    \begin{equation}
%         \pi_* \frac{\mathcal B_{m+1}\left(\frac{c_1\ca L}{r}, c_1\omega\right)}{(m+1)!}+
%{r} \sum_{0\le a<r} j_{a, *}
%\frac{[\mathcal B_{m+1}(\frac1r({a}\psi+ (i_a)^*(c_1\ca L),\psi)]_{\deg_\psi>1}}{(m+1)!\psi(\psi+\psi')}, 
%    \end{equation}
%    see \ref{rem:ch_formula_withBH}. 
%\end{remark}

\subsubsection{Pixton's formula}

The usual translation between the Chern character and the total Chern class \ref{eq:chern_char_class} allows us to convert \ref{ch_formula_strata} into an expression for $r^{u+1}\epsilon_*c_{g+u}(-R\pi_*\ca L^{1/r})$, but the result is not particularly attractive. However, in the case $u=0$, it is possible to write the constant term in $r$ in a much cleaner way (we do not know whether such a re-writing is possible for higher $u$). For each positive integer $r$ we define a class
\begin{equation}\label{eq:Pixton_strata_version}
\begin{split}
& P(r) \coloneqq \left(\exp{  -\frac{1}{2}\pi_* (c_1\ca L)^2} \right)\cdot \\
& 
 \sum_{\tilde \Gamma \in \rgraphs}  \frac{r^{- h^1(\Gamma)}}{\abs{\on{Aut}(\tilde \Gamma)}} \sum_{w\in W_r(\tilde \Gamma)} {j_{\tilde\Gamma}}_* \left[ \prod_{\{h, h'\}\in E(\tilde \Gamma)} \frac{1-\exp{( -\frac{1}{2}(\frac{w(e)}{r})^2 \psi_h + \psi_{h'}})}{\psi_h + \psi_{h'}} \right]
 \end{split}
\end{equation}
in the Chow ring of $\PIC$, where $\rgraphs$ is the set of pairs $(\Gamma, \delta)$ of a graph and a `multidegree' function, and  $W_r(\tilde \Gamma)$ is the set of weightings modulo $r$ on $\Gamma$ balancing $\delta$. We define $P_d(r)$ to be the codimension-$d$ part of this expression. 

\begin{proposition}\label{preintrotheorem:Pixton} We restrict to the total-degree-zero part of the Picard stack. 
Fix a positive integer $d$. Then the constant terms of the polynomials in $r$ given by 
    \begin{equation}
r^{2d+1-2g}    \epsilon_* c_d(-R\pi_*\ca L^{1/r})  \text{ and } P_d(r)
    \end{equation}
are equal in the Chow ring of $\PIC^{\mathrm{tot}0}$. 
\end{proposition}

\subsection{Applications}

\subsubsection{Proof of the Hodge-DR Conjecture}\label{sec:hodge_DR}

The Hodge-DR conjecture \cite[Conjecture 1.4]{Chen2022A-tale-of-two-m} predicts that, given non-negative integers $g, k$, and $u$, and a vector of integers $(a_1, \dots, a_n) \in \bb Z^n$ with sum $k(2g-2 + n)$, the class
\begin{equation}
\DR\left(\omega_{\log}^{\otimes k}\left(-\textstyle\sum_i a_i x_i\right)\right) \eta^u \in \Chow^{g+u}(\Mbar_{g,n})
\end{equation}
is equal to the coefficient of $r^u$ in the eventually-polynomial expression $r^{2u+1}\epsilon_*c_{g+u}(-R\pi_*\ca (\omega_{\log}^{\otimes k}(-\sum_i a_i x_i))^{1/r})$. This conjecture was proven for $u=0$ in \cite{Bae2020Pixtons-formula}, and for $k=0$ in \cite{Chen2022A-tale-of-two-m}. The general case now follows as a special case of \ref{preintrotheorem:DR}. 

We also remark that, combining the case $g=d$ of \ref{preintrotheorem:Pixton} with the case $u=0$ of \ref{preintrotheorem:DR}, we recover the main theorem of \cite{Bae2020Pixtons-formula}, and after suitable pullback also the main results of \cite{Janda2016Double-ramifica} and \cite{Janda2018Double-ramifica}. However, this does not yield new proofs, in the sense that we again reduce to a localisation computation. 

\subsubsection{Statement and proof of the log-Hodge-DR Conjecture}
In the same spirit, \ref{preintrotheorem:logDR} can be seen as a version of the Hodge-DR conjecture (and a proof) for the log DR cycle.

\subsubsection{Generalising the logDR formula}

Combining \ref{thm:intro_ch_formula} with the $u=0$ case of \ref{preintrotheorem:logDR} gives a new proof of the main result of \cite{Holmes2022Logarithmic-double}, after applying the method of that paper to re-write \ref{eq:Pixton_strata_version} in terms of piecewise polynomial functions on the moduli space of tropical curves. This is a genuinely different proof, in that it runs via the Thom--Porteous formula rather than via localisation. In particular, it applies to more general stability conditions, and is valid over fields of arbitrary characteristic (as long as $r$ is chosen coprime to the characteristic). 

\subsubsection{Extending results of Pagani, Ricolfi and van Zelm}

For a line bundle $\ca L$ on the universal stable curve over $\Mbar_{g,n}$, \cite{Pagani2020Pullbacks-of-un} give a very explicit formula for the Chern character of $R\pi_*\ca L$ in the Chow ring of $\Mbar_{g,n}$, together with applications to the computation of Brill-Noether cycles. Our \ref{thm:intro_ch_formula} can be seen as a generalisation of their results where we allow for more general bases and also for taking roots of $\ca L$, but it is also less explicit. In \ref{sec:formula_on_Mbar} we give a more explicit version of our formula in the case of roots of line bundles on the universal curve over $\Mbar_{g,n}$, generalising their results to roots.

\subsection*{Acknowledgements}
A.C. was supported by the
ANR-18-CE40-0009. D.H. is supported by grant VI.Vidi.193.006 of the Dutch Research Council (NWO). Additional support for this collaboration was provided by NUFFIC grant 2009008400: Spin double ramification cycles. 

%Dimitri Zvonkine's 
%seminars\footnote{``Cycles and Moduli'', ETH, 27-30 March 2017 and  
%2017-2018 Warwick EPSRC Symposium on Geometry, ``Topology and Dynamics in Low Dimensions''.}
%on his thoughts behind 
%the $r=0$ method are in many ways the starting point of the localised Chern class approach presented here. They have inspired us over these years and we are extremely grateful to him for his generosity in sharing his ideas with the community and for the helpful conversations that followed.

%We are very grateful to Jonathan Wise for helping us with the deformation theory in the heart of the decomposition in Step 3. 

We also would like to thank Felix Janda, Sam Molcho, Aaron Pixton, Kamyar Amini, Jonathan Wise, and Johannes Schmitt for helpful comments and discussion.

\section{Definitions and detailed statements of results}

In this section we present background material in logarithmic geometry and in tautological classes on the universal Picard stack, in order to give precise statements of our main results Theorems \oref{introtheorem:DR}, \oref{introtheorem:logDR}, and \oref{introtheorem:ch_formula}. %, and \oref{introtheorem:GRR}. 
%After setting up some conventions, we give some background in logarithmic geometry, and then in tautological classes on the universal Picard stack. 

\subsection{Conventions}

We work over a field of arbitrary characteristic, except in \ref{sec:proof_of_hodge_DR} where we restrict to characteristic 0 in order to apply the localisation results of \cite{Fan2019Higher-genus-re}. Whenever we consider spaces of $r$th roots of a line bundle, we always assume that $r$ is chosen coprime to the characteristic. We work with Chow rings with rational coefficients. 

We write $\mathfrak M$ for the stack of prestable stable curves. We write $\PIC$ for the stack of pairs of a prestable curve and a line bundle (the \emph{universal Picard stack}), and $\PIC^{\mathrm{tot}0}$ for the open substack where the line bundle has fibrewise total degree 0 (i.e. the sum of the degrees on the irreducible components of the fibre is 0). When we work with Chow rings of (possibly non-finite type) Artin stacks, we always mean the Chow cohomology with test objects separated finite-type schemes. This is a slight restriction of the theory in \cite{Bae2020Pixtons-formula} where separatedness is not required, which we make in order to apply Riemann-Roch. It seems plausible to us that the natural map between these two theories is an isomorphism (and indeed this is so if the target is a separated scheme, as explained to us by Bae), but we do not have a proof in the general case. 

% \david{New sentence above; does it look OK? }

If $\pi\colon C \to S$ is a family of prestable curves, there are two reasonable notions of generalised Jacobian. We write $\Jac_{C/S}$ for the relative coarse moduli space of line bundles which are fibrewise of total degree 0, and $\Jac^{\ul 0}_{C/S}$ for the open substack of those bundles which have multidegree $\ul 0$, in other words degree 0 on every irreducible component of every fibre. These Jacobians coincide exactly where the fibres of $C \to S$ are irreducible. The space $\Jac^{\ul 0}_{C/S}$ is a semiabelian algebraic space over $S$, in particular it is separated; in contrast, $\Jac_{C/S}$ often fails to be separated. 

\subsection{Roots of line bundles on curves}

For a positive integer $r$, we write $\frak M[r]\to \frak M$ for the root stack along all boundary divisors (the inertia at a point is of order equal to $r$ raised to the power the number of nodes). This carries a universal $r$-twisted curve $\ca C$. If $\pi\colon C \to S$ is a prestable curve inducing a map $S \to \frak M$, we write $S[r] = S \times_\frak M \frak M[r]$ , and denote the corresponding twisted curve again by $\ca C$. If $\ca L$ is a line bundle on $C$ whose fibrewise-degree is divisible by $r$, we write $\ca S^{1/r}$ for the stack $r$th roots of $\ca L$ on $\ca C$, with natural map $S^{1/r} \to S[r] \to S$ denoted by $\epsilon_r$. The twisted curve $\ca C\times_{S[r]} S^{1/r}$ carries a universal $r$th root $\ca L^{1/r}$. 

%\david{I just wrote the tiny section above, do we already say it elsewhere? Could not find it... } \ale{I think it is perfect.}

\subsection{Logarithmic background}
\label{sec:log_background}
%This section, which can be skipped at a first reading, provides a precise construction of the crucial modification $\tilde S$ used in \ref{eq:logDR_def} to define the DR locus and DR cycle. The constructions given here do not require that $S$ be smooth or that $C$ be smooth over a dense open of $S$, but do require the use of log geometry. 

To fix notation, a log scheme is a triple $(X, \M_X, \alpha)$ where $X$ is a scheme, $\M_X$ is a sheaf of monoids on $X_\et$, and $\alpha\colon \M_X \to \ca O_X$ is a map of sheaves of monoids (where $\ca O_X$ has its multiplicative monoid structure) such that $\alpha^{-1}\ca O_X^\times \to \ca O_X^\times$ is an isomorphism. We write $\ghost_X = \M_X /\M_X^\times$ for the \emph{ghost sheaf} --- the quotient of $\M_X$ by the subsheaf of units. We work throughout with fine saturated log structures. Details of log geometry can be found in \cite{Ogus2018Lectures-on-log}, and log curves are introduced in \cite{Kato2000Log-smooth-defo}. We learnt the approach to piecewise linear functions described below from \cite{Marcus2017Logarithmic-com}; it is explained in further detail in \cite{Chen2022A-tale-of-two-m}. 

\subsubsection{Log curves}
A log curve is a proper log smooth integral vertical\footnote{This means that the log structure does not see the markings, more precisely that it is strict on the smooth locus of $\pi$. } morphism $\pi\colon C \to S$ whose underlying map on schemes is a prestable (unmarked) curve. If $C \to S$ is the universal curve over the stack of (pre)stable curves, then equipping both $C$ and $S$ with their divisorial log structures coming from their normal-crossings boundary divisors gives an example of a log curve structure. This is called the \emph{minimal} log structure, and every other log curve structure is obtained by pulling back this one.  

%Every prestable curve admits the structure of a log curve; there is a \emph{minimal} log structure which on the base $S$ is given by taking the strict pullback of the log structure on the stack of prestable curves induced by the boundary divisor, and every other log curve structure is pulled back from this one. 

%We equip $C/S$ with its minimal log structure as a log curve; this implies that at each point $s$ of $S$, the rank of the characteristic monoid $\ghost_{S,s}$ is equal to the number of nodes in the fibre $C_s$. \david{minimal or any? I think any is better. Ale writes:  ``I myself would appreciate that we insist on the fact that there is the canonical (minimal?) log structure. My understanding is not so good yet, but I think I understand that it depends on the loci where the nodes persist, on the thickness of the nodes (the local equation $xy=z^t$) and on the markings. I know this is the ABC of log curves and maybe even wrong. But a sentence saying this (and maybe specifying to which extent it is not true) would be so illuminating. ''}

Let $\pi\colon C \to S$ be a log curve, and let $s$ be a geometric point of $S$. The graph $\Gamma_s$ of $C_s$ acquires a \emph{metric} sending each edge to a non-zero element of $\ghost_{S,s}$, which we call the \emph{length}. This can be defined by setting the length of an edge $e$ corresponding to a singular point $c \in C_s$ to the unique $\ell_e \in \ghost_{S, s}$ such that 
\begin{equation}
\ghost_{C,s} \cong \{(a,b) \in \ghost_{S,s}^2 : \ell_e \mid (a-b)\}
\end{equation}
where $\ell_e \mid (a-b)$ means that there exists an integer $s$ such that $a-b = s\ell_e$. The image $\alpha(\ell_e)$ of $\ell_e$ in $\ca O_{S,s}$ gives a local equation for the node, in the sense that the node is locally cut out by $xy = \alpha(\ell_e)$; if $C$ is smooth over a schematically dense open of $S$ then this equation determines $\ell_e$, but not in general. In this way, the length $\ell_e$ can be seen as a generalisation of the notion of the \emph{thickness} of a singularity in the sense of rigid analytic geometry. %if we want a citation for this, a standard choice is:  \url{http://www.numdam.org/item/10.24033/msmf.170.pdf}

\subsubsection{Piecewise linear functions}\label{sec:PL_functions}
A \emph{piecewise linear (PL)} function on $C$ is a section of the groupified ghost sheaf $\ghost_C^\gp$. We refer the reader to \cite[Lemma 2.12]{Chen2022A-tale-of-two-m} for a relation to the perhaps more familiar notion of a PL function on a tropical curve. 

The short exact sequence 
\begin{equation}
1 \to \ca O_C^\times \to \M_C^\gp \to \ghost_C^\gp \to 0
\end{equation}
induces a map from PL functions on $C$ to line bundles on $C$, which we denote by $\beta \mapsto \ca O_C(\beta)$.

We can see $\ca O_C(\beta)$ as a generalisation of the notion of the line bundle defined by a vertical Cartier divisor on $C$ (a Cartier divisor which is contracted by the map $\pi\colon C \to S$). To see this, suppose for a moment that $C \to S$ is smooth over a dense open subscheme of $S$, and that the log structure is minimal. Then the line bundle $\ca O_C(\beta)$ has a canonical rational section, inducing a vertical Cartier divisor on $C$. In general such a canonical rational section does not exist, yet the construction of $\ca O_C(\beta)$ always makes sense.

By the definition of $C/S$ being a log curve, at each smooth point $c$ of $C$ there is a canonical isomorphism 
\begin{equation}
\ghost_{C, c} \iso \ghost_{S, \pi(c)}, 
\end{equation}
so that we can think of a PL function $\beta$ as assigning to each irreducible component of each fibre $C_s$ (equivalently, each vertex of $\Gamma_s$) an element of $\ghost_{S,s}$; the \emph{value} of $\beta$ at that irreducible component. These values then satisfy the condition that $\ell_e \mid (\beta(u) - \beta(v))$ if $e$ is an edge from $u$ to $v$; we call 
\begin{equation}
\frac{\beta(v) - \beta(u)}{\ell_e} \in \bb Z
\end{equation}
the \emph{slope} of $\beta$ from $u$ to $v$. 

If $M$ is an integral monoid (i.e. $M$ injects into $M^\gp$), we have a partial order on $M^\gp$ by setting $m \ge n$ if $m - n \in M \subseteq M^\gp$. 

Given a geometric point $s$ of $S$ and a PL function $\beta$ on $C_s$,  we say $\beta$ is \emph{totally ordered} if for every two irreducible components $u$ and $v$ of $C_s$, either $\beta(u) \le \beta(v)$ or $\beta(v) \le \beta(u)$. We say $\beta$ has \emph{minimum value 0} if there exists a component $v$ with $\beta(v) = 0$, and if for every component $u$ we have $\beta(u)\ge 0$. %The intuition behind this terminology is that we put a partial order on $\ghost_{S,s}^\gp$ by declaring the elements of $\ghost_{S, s}$ to be non-negative. 

Suppose that $\beta$ is totally ordered and takes minimum value $0$. We say $\beta$ is \emph{reduced} if whenever $e$ is an edge from $v_1$ to $v_2$ with slope $s$ and $u$ is a vertex with $\beta(v_1) < \beta(u) < \beta(v_2)$, the difference $\beta(v_2) - \beta(u)$ is divisible by the slope $s$. This ensures that a certain fibre product of log curves is reduced, see \cite[\S 6]{Bae2020Pixtons-formula} for details. %\david{Do we actually need reducedness in this paper? I'm not sure. We talk about the resolution in the intro, but do we actually care? If not then we can simplify things. If so then we should explain this a little better. }

%\ale{Saw your commented doubt about the need of this here. Indeed, this notion got me into 
%lots of thinking of why would you like to 
%say that for ANOTHER component possibly far away, in case its weight fits between the  weights on the two branches of a note (I'm thinking in terms of weights, ie vertical divisors) than it should be compatible with the slope of the node-edge. Curious really curious to see with my eyes the reason for this. But probably not relevant in this paper. }\david{Think about mapping the whole curve to a chain of rational curves (in some sense this is what $\beta$ does). Suppose you want to insert an extra rational curve into the middle of the chain, and have a map to this from some destabilisation of your curve. If you want to do this in families, this imposes relations among the thicknesses of the edges into which you want to insert new vertices, even if these edges are `far apart'. But indeed, not really relevant to discuss for this paper I think. }
%\ale{I think this can be removed.}

\subsection{Computing DR}\label{sec:computing_DR}
Let $\pi\colon C \to S$ be a log curve and $\ca L$ a line bundle on $C$ fibrewise of total degree $0$. 
The logarithmic DR locus $\logDRL(\ca L)$ represents the functor %\ale{would it be possible to avoid this kind of huge formula, to a new reader they are unreadable, and to me they are so scary that I often forgot them after deciphering them}\david{better now? }
\begin{align}\label{eq:logDRL}
\cat{LogSch}_S^\op &\to \cat{Set}; \nonumber \\ T &\mapsto \{ (\beta, \ca F, \phi)\}% : \beta  \in \sf{PL}(C_T), \ca F \in \Pic(T), \phi\colon \ca L \iso \ca O_C(\beta) \otimes \pi^*\ca F\}
\end{align}
where $\beta$ is a reduced piecewise linear function on $C_T$, $\ca F$ is a line bundle on $T$, and $\phi\colon \ca L \iso \ca O_C(\beta) \otimes \pi^*\ca F$ is an isomorphism. For a given $T$, such a triple $(\beta, \ca F, \phi)$ is unique up to unique isomorphism if it exists, so $p\colon \logDRL(\ca L )\to S$ is a monomorphism in the logarithmic category. We denote by $\polarbdl$ the \emph{tautological} line bundle $\ca F^\vee$ on $\logDRL(\ca L)$, with first Chern class $\polarcls$.

An alternative construction of $\polarbdl$ is available if we can find a sufficient ample effective horizontal divisor $D$ on $C$; then the maps
\begin{equation}
\pi^*\ca F \iso \ca L(-\beta) \hookrightarrow \ca L \to \ca L|_D
\end{equation}
induce by adjunction a map $\ca F \to \pi_*\ca L|_D$, and hence a map 
\begin{equation}\label{rem:other_embedding}
\logDRL(\ca L) \to \bb P(\pi_*\ca L|_D)
\end{equation}
with the property that $\ca F^\vee$ is the pullback of $\ca O(1)$ along this map (see \cite[Lemma 6.1]{Chen2022A-tale-of-two-m}, \cite[\href{https://stacks.math.columbia.edu/tag/0FCY}{Example 0FCY}]{stacks-project}; the dual comes from the fact that we define the projectivisation via subbundles rather than quotient bundles).

%An alternative construction of $\polarbdl$ is available if we can find an effective divisor $D$ on $C$ such that $\pi_*\ca L(D)$ is a vector bundle; then the maps
%\begin{equation}
%\pi^*\ca F \iso \ca L(-\beta) \hookrightarrow \ca L \hookrightarrow \ca L(D)
%\end{equation}
%induce by adjunction a map $\ca F \to \pi_*\ca L(D)$, which is a map 
%\begin{equation}\label{eq:logDR_to_proj}
%\logDRL(\ca L) \to \bb P(\pi_*\ca L(D))
%\end{equation}
%It is easily seen \cite[Lemma 6.1]{Chen2022A-tale-of-two-m} that $\ca F^\vee$ is the pullback of $\ca O(1)$ along this map. 

%\begin{remark}\label{rem:other_embedding}
%There is a also a natural restriction map $\ca  L \to \ca L|_D$, and the same construction using adjunction gives a map $\ca F \to \pi_*\ca L_D$, and hence a map 
%\begin{equation}
%\logDRL(\ca L) \to \bb P(\pi_*\ca L|_D)
%\end{equation}
%again with the property that $\ca F^\vee$ is the pullback of $\ca O(1)$ along this map (see \cite[\href{https://stacks.math.columbia.edu/tag/0FCY}{Example 0FCY}]{stacks-project}; the dual comes from the fact that we define the projectivisation via subbundles rather than quotient bundles). 
%\end{remark}

In order to equip $\logDRL(\ca L)$ with a virtual fundamental class, we first embed it in an ambient space $\tilde S^\ca L$, defined as the space representing the functor 
\begin{align}\label{eq:StildeL}
\cat{LogSch}_S^\op &\to \cat{Set}; \nonumber \\ T &\mapsto \{ \beta \in \sf{PL}(C_T) : \mdeg \ca L  = \mdeg \ca O_C(\beta)\}; 
\end{align}
here $\beta$ is exactly as in the definition of $\logDRL(\ca L)$, the only difference being that now we simply require the line bundle $\ca L(-\beta)$ to have multidegree $\ul 0$ on each fibre (rather than imposing that it be a pullback from the base). There is a natural closed immersion 
\begin{equation}
\logDRL(\ca L) \hookrightarrow \tilde S^\ca L, 
\end{equation}
and the map $\tilde S^\ca L \to S$ is a log monomorphism and is an isomorphism over the locus of irreducible curves. 
%\ale{So it is really me... I don't see why it is an isomorphism over the locus of irreducible curves. DR is not an isomorphism to Pic0... I'm sure I'm missing the point.}\david{If we stare at equation (2.3.2) for an irreducible curve, then we have to tale $\beta = 0 $ (since $\beta$ takes only one value, and it is required to take minimum value 0). Then we are just asking that the (multi) degree of $\ca L$ be the same as the (multi)degree of the trivial bundle, which is automatic for an irreducible curve. It's actually also an iso over the compact-type locus, more generally the locus of treelike curves, but this needs a little extra argument, and not sure if it is worth pointing out aw we do not use it. } 
If $S = \Mbar_{g,n}$ then $\tilde S^\ca L \to S$ is an open subscheme of a modification of $S$ along the boundary (i.e. a proper birational morphism which is an isomorphism away from the boundary); in general it is a base-change of such a map. There is a natural map 
\begin{equation}
\tilde S^\ca L \to \Jac^{\ul 0}_{C/S}; \beta \mapsto \ca L(- \beta), 
\end{equation}
and the pullback of the unit section is exactly $\logDRL(\ca L)$. These constructions all commute with base-change, and so this situation is pulled back from the universal situation where $S = \PIC^{\mathrm{tot}0}$ and $\ca L$ is the universal line bundle, in which case $\tilde S^\ca L$ is smooth (see \cite[Theorem 2.4]{Chen2022A-tale-of-two-m}), hence the map $\tilde S^\ca L \to \Jac^{\ul 0}_{C/\PIC^{tot0}}$ is l.c.i. We define the virtual fundamental class $\logDR(\ca L)$ of $\logDRL(\ca L)$ to be the Gysin pullback of the unit section along this map. 

\begin{definition}\label{def:DR}
Given $u \in \bb Z_{\ge 0}$, we define 
\begin{equation}
\DR(\ca L)[u] = p_*(\logDR(\ca L) \cdot \polarcls^u) \in \Chow^{g+u}(S)
\end{equation}
where $p\colon \logDRL(\ca L) \to S$, and write $\DR(\ca L) = \DR(\ca L)[u]$. If we omit the $\ca L$ we refer to the universal version where $S = \PIC^{tot0}$ and $\ca L$ is the inverse line bundle. 
\end{definition}
In particular, if $\phi_\ca L\colon S \to \PIC^{tot0}$ classifies $\ca L$ then $\phi^*\DR[u] = \DR(\ca L)[u]$. 

 With this notation we can state a more precise version of \ref{preintrotheorem:DR}. 
%\ref{preintrotheorem:DR}. 
\begin{introtheorem}\label{introtheorem:DR} 
For all $u \ge 0$, the equality 
\begin{equation}
\DR(\ca L)[u] = [r^{u+1}\epsilon_*c_{g+u}(-R\pi_*\ca L^{1/r})]_{r=0}
\end{equation}
holds in $\Chow^{g+u}(S)$, where $\polarcls$ is the first Chern class of the tautological line bundle, and $[\ \ ]_{r=0}$ indicates taking the $r$-constant term of an eventually-polynomial expression in $r$. In particular by taking $u=0$ we obtain
\begin{equation}
\DR(\ca L) = [r\epsilon_*c_{g}(-R\pi_*\ca L^{1/r})]_{r=0}. 
\end{equation}
\end{introtheorem}

%\david{I think that for compatibility with what we write later, it might be convenient to write $\DR[u]$ for the universal version (on $Pic)$. TBC. }

%In \ref{sec:intro_RR} we will explain how to compute the right hand side. The $u=0$ case gives new proofs of the main results of \cite{Janda2016Double-ramifica,Janda2018Double-ramifica,Bae2020Pixtons-formula}, and the case of higher $u$ resolves the \emph{Hodge-DR} conjecture from \cite{Chen2022A-tale-of-two-m}. 

\subsection{Computing logDR}\label{intro:logDR}

Suppose we wish to compute the class of $\logDR(\ca L)$ itself, without pushing forward to $S$. It is not immediately clear what this means, since $\logDRL(\ca L)$ is not contained in $S$, and so it does not make sense to try to write $\logDR(\ca L)$ in terms of classes on $S$. A solution to this, introduced in \cite{Holmes2017Extending-the-d,Holmes2017Multiplicativit}, is to work in the \emph{logarithmic Chow ring} of $S$: the colimit of the Chow rings of suitable log modifications\footnote{A modification of schemes is a proper birational morphism. A log blowup is a blowup in a sheaf of monoid ideals. A log modification is a proper representable log monomorphism which can be refined by log blowups. Informally, a log modification can be thought of as an iterated blowup in boundary strata. } of $S$. In essence, this means that we embed $\logDRL(\ca L)$ in some sufficiently fine log blowup $\tilde S \to S$, compute the class of $\logDR(\ca L)$ on $\tilde S$, and can then guarantee that for any other choice of log blowup of $S$ the resulting classes will coincide after pullback to a common refinement. Further discussion of the log Chow ring can be found in \cite{Holmes2017Multiplicativit,Holmes2021Logarithmic-int,Molcho2021The-Hodge-bundl,Botero2021Chern-Weil-and-,Dang2022Intersection-th,Molcho2021A-case-study-of,Holmes2022Logarithmic-double}. 

We are primarily interested in the case where $C \to S$ is smooth over a dense open; in this case the log modification are just a special class of (schematic) modifications, in particular they are proper and birational. However, in general log modifications need not be birational, and setting up the theory of log Chow rings correctly in this context is quite delicate, see \cite{Barrott2019Logarithmic-Cho}. We sidestep these technicalities via the observation that, if $\tilde S \to S$ is a log modification and $i\colon \logDRL(\ca L) \to \tilde S$ is a factorisation of $p\colon \logDRL(\ca L) \to S$ such that $i$ is \emph{strict}\footnote{The log structure on $\logDRL(\ca L)$ induced by \ref{eq:logDRL} coincides with the pullback log structure from $\tilde S$. }, then $i$ is a closed immersion and the class $i_*\logDR(\ca L)$ determines the logDR cycle in the log Chow ring. 

We will construct such factorisations $\logDRL(\ca L) \to \tilde S \to S$ using the theory of \emph{compactified Jacobians} in \ref{sec:proof_of_thm_logDR}. For now we note that such a factorisation \emph{exists} whenever $C \to S$ is either stable or admits a section through the smooth locus, and that the space $\tilde S$ carries a `quasi-stable model' $\tilde C \to C\times_S \tilde S$ obtained by inserting rational curves at some nodes, and a `quasi-stable line bundle' $\tilde{\ca L}$ which differs from $\ca L$ by a twist by a piecewise linear function. 

This line bundle $\tilde{\ca L}$ has two crucial properties: it is fibrewise trivial \emph{exactly} on the image of $\logDRL(\ca L) \to \tilde S$, and it admits a non-zero global section on a fibre \emph{if and only if} it is trivial on the fibre. This simplifies greatly the comparison to a suitable degeneracy locus, so that it is no longer necessary to take a constant term in $r$. We state a precise version of \ref{preintrotheorem:logDR}. 

\begin{introtheorem}\label{introtheorem:logDR}
For all $u \ge 0$, for all $r \ge 1$, we have
\begin{equation}
\logDR(\ca L) \cdot \polarcls^u = r^{u+1}\epsilon_*c_{g+u}(-R\pi_*\tilde {\ca L}^{1/r}).
\end{equation}
\end{introtheorem}
%In \ref{sec:intro_RR} we will explain how to compute the right hand side, yielding a new proof of the main results of \cite{Holmes2022Logarithmic-double}. 

An immediate consequence of this theorem is that, for fixed $u$, the classes $r^{u+1}\epsilon_*c_{g+u}(-R\pi_*\tilde{\ca L}\ca F^{1/r})$ form a \emph{constant} polynomial in $r$; all coefficients of positive powers of $r$ vanish. This potentially-fruitful source of relations in the log tautological ring has yet to be explored.

\subsection{Riemann--Roch computations on the Picard stack}
\label{sec:intro_RR}

%\david{Should we move these theorems earlier, to where they can first be stated? This will change the order a bit, but I think that's not necessarily bad? We then need a section at the end of the intro where we explain how to recover lots of other people's theorems from this.  }

%The following theorems hold for any prestable $C/S$ and any line bundle $\ca L$ on $C$ fibrewise of total degree 0. 

We have seen that expressions of the form $r^u\epsilon_*c_{g+u}(-R\pi_*\ca L^{1/r})$, and their constant terms in $r$, allow us to determine (log) double ramification cycles. It remains to see how we can compute these classes. If $S = \Mbar_{g,n}$ then we can write these expressions in terms of the standard tautological classes, but for more general $S$ we need a more flexible framework, which we take from \cite{Bae2020Pixtons-formula}. The idea is to view the pair $(C, \ca L)$ as being equivalent to a map from $S$ to the universal Picard stack $\PIC$. Similarly, twisted $r$th roots of $\ca L$ determine a map from $S^{1/r}$ to the stack $\epsilon\colon \PIC^{1/r} \to \PIC$ of twisted $r$th roots of the universal line bundle. 
We write $\pi\colon \ca C \to \PIC^{1/r}$ for the universal $r$-twisted curve, and $\ca L^{1/r}$ for the universal $r$th root of the universal bundle on $\PIC$. 

% \david{Here's a definition of the space of rigidified $r$th roots: }
% \begin{definition}[Rigidified $r$th roots]
% To rigidify the stack of $r$th roots, we fix a section $x$ of $C/S$. Then $\PIC^{1/r}$ has objects 
% \begin{equation}
% (\ca C/S, x, \ca L, \ca F, z\colon x^*\ca F \iso \ca O_S, \phi\colon \ca F^{\otimes r} \iso \ca L \otimes \pi^*x^*\ca L^\vee)
% \end{equation}
% where 
% \begin{enumerate}
% \item $\pi\colon \ca C \to S$ is an $r$-twisted curve with relative coarse space $C$; we write $x\colon S \to \ca C$ for the unique lift; 
% \item $\ca F$ is a line bundle on $\ca C$. 
% \end{enumerate}
% \end{definition}
% \david{There is a minor annoyance that we have to choose a section, and a potentially more serious annoyance that we are not taking roots of $\ca L$, but rather of the rigidification $\ca L \otimes \pi^*x^*\ca L^\vee$. }

Strata of $\PIC^{1/r}$ are indexed by triples $(\Gamma, \delta, w)$ where 
\begin{enumerate}
\item
$\Gamma$ is a prestable graph of genus $g$ in the sense of \cite[0.3.1]{Bae2020Pixtons-formula}, with vertices $V(\Gamma$), edges $E(\Gamma)$, and half-edges $H(\Gamma)$;
\item $\delta\colon V(\Gamma)\to \bb Z$ is a function;
\item $w\colon H(\Gamma) \to \{0, \dots, r-1\}$ is a \emph{weighting modulo $r$ balancing $\delta$}; in other words, for each edge $e = \{h, h'\}$ we have $w(h) + w(h') = 0 \mod r$, and for each vertex $v$ we have $\sum_{h \in H(v)} w(h)  =\delta(v)\mod r$, where $H(v)$ is the set of half-edges attached to the vertex $v$. 
\end{enumerate} 
We write $\rgraphs^{1/r}$ for the set of such triples (shortened to $\rgraphs$ if $r=1$). Given also an integer $i$, we write $\rgraphs ^{1/r}(i) \subseteq \rgraphs^{1/r}$ for the subset where the graph has exactly $i$ edges. Given $(\Gamma, \delta, w) \in \rgraphs^{1/r}$ we construct the stratum $  \PIC_{\Gamma, \delta, w}^{1/r}$ (of codimension equal to the number of edges of $\Gamma$) via a commutative diagram
\begin{equation}
 \begin{tikzcd}
  \PIC_{\Gamma, \delta, w}^{1/r} \arrow[r, hook, "2"] \ar[rrr, bend left, "j_{\Gamma, \delta, w}"]& \PIC_{\Gamma, \delta}^{1/r} \ar[r, hook, "1"] & \PIC^{1/r}_\Gamma \ar[r] \ar[d] & \PIC^{1/r}\ar[d] \\
  &&\mathfrak M_\Gamma[r]\ar[r] & \mathfrak M[r];\\
\end{tikzcd}
\end{equation}
 here $\mathfrak M_\Gamma[r]$ is the underlying reduced substack of $\mathfrak M[r] \times_{\mathfrak M} \mathfrak M_\Gamma = \prod_v \mathfrak M_{g(v), H(v)}$, the square is a pullback, the arrow (1) is the inclusion of the connected component where the multidegree of $\ca L$ specialises to $\delta$, and the arrow (2) is the inclusion of the connected component where the characters of $\ca L^{1/r}$ at nodes are given by $w$. The codimension of the image of $\PIC_{\Gamma, \delta, w}^{1/r}$ in $\PIC^{1/r}$ via the finite representable map $j_{\Gamma, \delta, w}$ is the number of edges of $\Gamma$, and the degree over the image is $\abs{\on{Aut} (\Gamma, \delta, w)}$. 

 We define analogues of the normally decorated strata classes of \linebreak \cite{Molcho2021The-Hodge-bundl}. 
 For $(\Gamma, \delta, w) \in \rgraphs^{1/r}$, we consider the ring $\bb Z[\psi_h: h \in H(\Gamma)]\otimes \bb Z[\xi_e: e \in E(\Gamma)]$, and define a map from this ring to $\Chow(\PIC^{1/r}_{\Gamma, \delta, w})$ sending $\psi_h$ to the pullback of $c_1(p_h^*\omega_\pi)$ from $\mathfrak M_{g(v), H(v)}$ where $h \in H(v)$ and $p_h\colon \PIC_{\Gamma, \delta, w} \to \ca C$ corresponds to $h$, and sending $\xi_e$ to $c_1(p_e^*\ca L^{1/r})$ where $p_e\colon\PIC^{1/r}_{\Gamma, \delta, w} \to \ca C$ corresponds to $e$. Given a polynomial $f$ in the $\psi_h$ and $\xi_e$, we abusively identify it with its image in $\Chow(\PIC^{1/r}_{\Gamma, \delta, w})$. We can think of the class $(j_{\Gamma, \delta, w})_*f$ as a tautological class on $\PIC^{1/r}$.

With this notation in hand we can expresses the Chern character of $R\pi_*\ca L^{1/r}$ as a  `bulk' term plus a sum of such classes over strata of codimension 1 (we write $\tilde \Gamma = (\Gamma, \delta, w) \in \rgraphs^{1/r}$ to shorten the notation); for the convenience of the reader we restate \ref{preintrotheorem:logDR} now that all terms are defined. 
 \begin{introtheorem}\label{introtheorem:ch_formula} 
 We have  \begin{multline}\label{ch_formula_strata} 
  \on{ch}_{m}(R\pi_*\ca L^{1/r})=
\pi_* \frac{\ca B_{m+1}\left(\frac{c_1\ca L}{r}, c_1\omega\right)}{(m+1)!} + \\
\sum_{\tilde \Gamma \in \rgraphs^{1/r}(1)} \frac{r}{\abs{\on{Aut}(\tilde \Gamma)}}\sum_{\substack{{p+q=m+1}\\{p\ge 2}}} 
 \frac{B_{p}(\frac{w(h)}{r})}{p!q!}(j_{\tilde \Gamma})_*\left((\xi_e)^{q} \frac{\psi_h^{p-1} - (-\psi_{h'})^{p-1}}{\psi_h + \psi_{h'}}\right), 
\end{multline}
where $\ca B_m(x,y) = y^mB_m(x/y)\in \bb Q[x,y]$ is the homogenised version of the usual Bernoulli polynomial $B_m(x)$, and $e = \{h, h'\}$ is the unique edge of $\Gamma$ (note that the summands a-priori depend on a choice between $h$ and $h'$, but the summand is invariant under changing that choice).  
\end{introtheorem}

%Then we observe that the above equation is just another 
%way to express the formula stated in
%\ref{thm:intro_ch_formula}.
%\end{remark}

\section{Invariance properties}\label{sec:invariances}

An important feature of the universal DR cycle $\DR = \DR[0]$ of \ref{def:DR} explored in \cite{Bae2020Pixtons-formula} is its `invariance' properties; behaviour under various pullbacks and shifts of indices. Here we prove analogues of these invariances for the ground classes $\DR[u]$ of \ref{def:DR}. For invariances II, III, and VI we also give analogous statements for the classes $c_\bullet(-\pi_! \ca L^{1/r})$, as they are used in the \emph{proof} of \ref{thm:uni_Hodge_DR} (our setup makes these trivial for II and III, but VI needs some work). Invariance VII is new, and does not make sense for $\DR$, but only for the classes $c_\bullet(-\pi_! \ca L^{1/r})$. We also use it in the proof of \ref{thm:uni_Hodge_DR}.

\subsection{The classes $\pcc_g^{r,e,d}$}
We first set up some notation for the classes $c_\bullet(-\pi_! \ca L^{1/r})$. Given positive integers $r$ and $d$, we define $\PIC_g^{1/r,d}$ to be the stack of pairs $(\ca C/S, \ca F)$ where $\ca C/S$ is an $r$-twisted prestable curve of genus $g$, and $\ca F$ is a line bundle on $\ca C$ such that $\ca F^{\otimes r}$ descends to a line bundle of degree $d$ on the relative coarse moduli space $C/S$ of $\ca C/S$. We write 
\begin{equation}
\epsilon\colon \PIC^{1/r, d}_{g} \to \PIC_{g}^d
\end{equation}
for the map sending $(\ca C/S, \ca F)$ to $(C/S, \ca F^{\otimes r})$.

Chern classes make sense in the Chow cohomology of $\PIC_g^{1/r,d}$. Moreover, the map $\epsilon\colon \PIC^{1/r, d}_{g} \to \PIC_{g}^d$ is proper, flat, lci, and relatively DM, and so we can push operational classes along it. \begin{definition}\label{def:omega_classes}
Given $r$, $d$ as above, and a non-negative integer $e$, we define 
\begin{equation}
\pcc_{g}^{r,  e, d} \coloneqq \epsilon_* c_e(-\pi_! \ca L^{1/r}) \in \Chow^e(\PIC_{g}^d)
\end{equation}
and
\begin{equation}
\pcc_{g}^{r, \bullet, d} \coloneqq \epsilon_* c(-\pi_! \ca L^{1/r}) \in \Chow^*(\PIC_{g}^d). 
\end{equation}
\end{definition}

%For the proof of \ref{thm:uni_Hodge_DR} we only need invariances II, III and VI, of which II and III are trivial, but VI needs a little work. However, we discuss each case as we suspect these may be of broader use. 
To state the invariances, we fix a little more notation. Let $C/S$ be a prestable curve, $p_1, \dots, p_n$ a collection of disjoint sections through the smooth locus, $\ca L$ on $C$ a line bundle of relative degree $d$, and $\ul a = (a_1, \dots, a_n)$ a vector of integers with sum $\frak a$. We write 
\begin{equation}
\phi_{\ca L, \ul a}\colon S \to \PIC_g^{d - \frak a}
\end{equation}
for the map classifying $\ca L(-\sum_i a_i p_i)$.

Our first four invariances are very straightforward, but we include them to facilitate comparison to the invariances listed in \cite{Bae2020Pixtons-formula}. 

\subsection{Invariance I: Dualising}
% \Dcomment{So far only for DR, not sure if there is a nice statement for the chern class}
Write $\ca L^\vee$ for the dual of $\ca L$, and set $- \ul a = (-a_1, \dots, -a_n)$. In the notation of \ref{eq:logDRL}, write $\beta^\sf{max}$ for the PL function on $\logDRL_g$ given by the maximum value that $\beta$ takes on vertices of the tropicalisation of the universal curve, which determines a boundary divisor $c_1(\ca O(\beta^\sf{\max}))$ on $\logDRL_g$. 
 
% \david{I think $\cat{Rub}_g$ is now just called $\logDRL$? }
\begin{lemma}Suppose $\frak a = d$. Then 
\begin{equation}
%\phi_{\ca L^\vee, -\ul a}^*\DR_{g}[u] = \phi_{\ca L, \ul a}^*((-c_1(\ca O(\beta^\sf{max})))^u \cdot \DR_{g}[u]). 
\phi_{\ca L^\vee, -\ul a}^*p_*(\logDR \cdot \polarcls^u) = \phi_{\ca L, \ul a}^*p_*(\logDR \cdot (-\polarcls+ c_1(\ca O(\beta^\sf{max}))^u)), 
\end{equation}
in particular if $u=0$ we recover 
\begin{equation}
\phi_{\ca L^\vee, -\ul a}^*\DR_{g} = \phi_{\ca L, \ul a}^*\DR_{g}. 
\end{equation}
\end{lemma}
\begin{proof}
The proof of \cite[Lemma 6.5]{Chen2022A-tale-of-two-m} goes through essentially unchanged in this more general context. 
\end{proof}

\subsection{Invariance II: Unweighted markings}
Let $p_{n+1}\colon S \to C$ be another section through the smooth locus, disjoint from $p_1, \dots, p_n$. Let $\ul a_0 = (a_1, \dots, a_n, 0)$. Then the maps 
\begin{equation}
\phi_{\ca L, \ul a}\colon S \to \PIC_g^{d - \frak a} \;\;\; \text{and} \;\;\; \phi_{\ca L, \ul a_0}\colon S \to \PIC_g^{d - \frak a}
\end{equation}
are tautologically equal, so that pulling back the classes $\pcc_{g,n}^{r,e,d}$ along either map gives the same result. 
If $\frak a = d$ then the same holds for $\DR_g[u]$. 

\subsection{Invariance III: Weight translation}
Let $\ul b = (b_1, \dots, b_n) \in \bb Z^n$, and define $\ca L_\ul b = \ca L(\sum_i b_i p_i)$. Then the maps 
\begin{equation}
\phi_{\ca L, \ul a}\colon S \to \PIC_g^{d - \frak a} \;\;\; \text{and} \;\;\; \phi_{\ca L_\ul b, \ul a + \ul b}\colon S \to \PIC_g^{d - \frak a}
\end{equation}
are tautologically equal, so that pulling back the classes $\pcc_{g,n}^{r,e,d}$ along either map gives the same result. 
If $\frak a = d$ then the same holds for $\DR_g[u]$. 

\subsection{Invariance IV: Twisting by pullback}
% \Dcomment{So far only for DR, not sure if there is a nice statement for the chern class}

Let $\ca B$ be a line bundle on $S$, and consider the map 
\begin{equation}
\phi_{\ca L\otimes \pi^*\ca B, \ul a}\colon S \to \PIC_g^{d - \frak a}
\end{equation}
classifying $\ca L\otimes \pi^*\ca B(-\sum_i a_i p_i)$. 

\begin{lemma}Suppose $d = \frak a$. Then 
\begin{equation}
\phi_{\ca L\otimes \pi^*\ca B, \ul a}^*\DR_g[u] = c_1(\ca B^\vee)^u\phi_{\ca L, \ul a}^*\DR_g[u]. 
\end{equation}
\end{lemma}
\begin{proof}
In the notation of \ref{eq:logDRL}, an isomorphism 
\begin{equation}
\pi^*\ca F(\beta) \cong \ca L \otimes \pi^*\ca B (-\sum_i a_i p_i)
\end{equation}
is equivalent to an isomorphism 
\begin{equation}
\pi^*(\ca F\otimes \ca B^\vee)(\beta) \cong \ca L (-\sum_i a_i p_i). \qedhere
\end{equation}
\end{proof}

\subsection{Invariance V: Vertical twisting}
% \Dcomment{So far only for DR, not sure if there is a nice statement for the chern class}
Consider a partition of the genus, marking data, and degree
\begin{equation}
g = g_1 + g_2, \;\;\; N_1 \sqcup N_2 = \{1, \dots, n\}, \;\;\; d_1 + d_2 = d
\end{equation}
which is not symmetric (i.e. $(g, N_1, d_1) \neq (g_2, N_2, d_2)$). This determines divisors $\Delta_1$, $\Delta_2$ in the universal curve over $\PIC_{g, n}^d$ corresponding to the $(g_1, N_1, d_1)$ component and the $(g_1, N_1, d_1)$ component of the curve over the separating boundary divisor corresponding to the partition. We use the same notation for the pullbacks of these divisors to $C$ along the classifying map of $\ca L$, yielding maps 
\begin{equation}
\phi_{\ca L(\Delta_1), \ul a}\colon S \to \PIC_g^{d - \frak a}\;\;\; \text{and}\;\;\;\phi_{\ca L(\Delta_2), \ul a}\colon S \to \PIC_g^{d - \frak a}
\end{equation}
Note that $\Delta_1 + \Delta_2$ is a pullback from $S$, and so the effect on $\DR_g[u]$ of twisting by this divisor is determined by Invariance IV. It therefore suffices to describe the result of twisting by one of the $\Delta_i$. 

\begin{lemma}
Suppose that $\frak a = d$, and that\footnote{Here we follow the sign conventions of \cite{Marcus2017Logarithmic-com}, which are the opposite of those in \cite{Chen2022A-tale-of-two-m}. }
\begin{equation}
d_1 + \sum_{i \in N_1} a_i \le d_2 + \sum_{i \in N_2} a_i. 
\end{equation}
Then 
\begin{equation}
\phi_{\ca L(\Delta_1), \ul a}^*\DR_g[u]= \phi_{\ca L, \ul a}^*\DR_g[u]. 
\end{equation}
\end{lemma}
\begin{proof}
In the notation of \ref{eq:logDRL}, suppose that 
\begin{equation}
\ca L\left(-\textstyle\sum_i a_i p_i\right) \cong \pi^* \ca F (\beta)
\end{equation}
with $\beta$ a PL function on $C$, totally ordered on vertices, and with minimum value $0$. Let $\delta$ be the PL function taking the value $0$ on the component corresponding to $\Delta_2$, and with slope $+1$ along the separating edge to the component corresponding to $\Delta_1$. Then 
\begin{equation}
\ca O_C(\delta) = \ca O_C(\Delta_1), 
\end{equation}
so
\begin{equation}
\ca L(\Delta_1)\left(-\textstyle\sum_i a_i p_i\right) \cong \pi^* \ca F (\beta + \delta), 
\end{equation}
and $\beta + \delta$ still has totally ordered values with minimum $0$ (perhaps after subdivision of $S$ to restore the total order, which will not affect the outcome after pushing down). 
\end{proof}

\subsection{Invariance VI: Partial stabilisation}

Suppose we are given a partial stabilisation of curves $C' \stackrel{f}{\to} C \stackrel{\pi}{\to} S$, and write $\ca L' = f^*\ca L$. The sections $p_i$ lift uniquely to sections $p_i'$ of $C'$, and we have a map 
\begin{equation}
\phi_{\ca L', \ul a}\colon S \to \PIC_{g}^{d - \frak a}
\end{equation}
classifying $\ca L'(-\sum_i a_i p_i')$. 

\begin{proposition}
In the Chow ring of $S$ we have equalities
\begin{enumerate}
\item if $d = \frak a$ then $\phi_{\ca L, \ul a}^*\DR_{g}[u] = \phi_{\ca L', \ul a}^*\DR_{g}[u]$; 
\item $\phi_{\ca L, \ul a}^* \pcc_{g}^{r, e} = \phi_{\ca L', \ul a}^* \pcc_{g}^{r, e}$. 
% \item if $d = \frak a$ then  $\phi_{\ca L, \ul a}^* \rpcc_{g}^{g+u}[r^u]= \phi_{\ca L', \ul a}^*\rpcc_{g}^{g+u}[r^u]$. 
\end{enumerate}
\end{proposition}
\begin{proof}
Claim (1) is proven by a minor variation on the proof of \cite[Lemma 58]{Bae2020Pixtons-formula}. We now prove claim (2).  Let $\tilde S \to S$ be the space of twisted $r$th roots of $\ca L$, which is naturally identified with the space of twisted $r$th roots of $\ca L'$. Let $\tilde {C'}\stackrel{\tilde f}{\to} \tilde C \stackrel{\tilde \pi}{\to} \tilde S$ be the universal twisted curves, with $\ca F$ and $\ca F'$ the respective $r$th roots; so $\ca F'  = \tilde f^*\ca F$. 

Since formation of $\pcc_{g}^{r, e}$ commutes with base-change, it suffices to construct an isomorphism 
\begin{equation}
R^\bullet \pi_*\ca F\iso R^\bullet(\tilde \pi \circ \tilde f)_*\ca F'. 
\end{equation}

By the Grothendieck spectral sequence, it is enough to show that 
\begin{equation}\label{eq:stabilisation_cohomology}
    R^\bullet \tilde f_* \ca F' = \ca F. 
\end{equation}

Working locally on $\tilde C$, we may assume that $\ca F$ (and hence $\ca F'$) is trivial. Choose an \'etale cover $U \to \tilde C$ by a scheme; the map $\tilde f\colon \tilde C' \to \tilde C$ is representable by \cite[Lemma 4.4.3]{abramovich2002compactifying}, hence $U' \coloneqq \tilde C' \times_{\tilde C} U$ is also a scheme. We can then apply \cite[\href{https://stacks.math.columbia.edu/tag/0E7L}{Tag 0E7L}]{stacks-project} to again reduce to the case where $\tilde S$ is a point. We are reduced to the following situation: $U = \Spec k[x,y]/(xy)$, and $f\colon U'\to U$ is obtained by inserting a chain of $r$ rational curves at the node of $U$. We need to show that 
\begin{equation}
    R^\bullet \tilde f_* \ca O_{U'} = \ca O_U. 
\end{equation}
This follows from an entertaining argument with the normalisation, which is surely well-known. Let $\tau\colon V \to U'$ denote the normalisation, so $V$ consists of two copies of $\bb A^1_k$ and $r$ copies of $\bb P^1_k$. Write $q_1, \dots, q_{r+1}$ for the nodes of $U'$, mapping to the node $q$ of $U$. Applying $Rf_*$ to the exact sequence
\begin{equation}
0 \to \ca O_{U'} \to \tau_*\ca O_V \to \bigoplus_{i=1}^{r+1} k_{q_i} \to 0
\end{equation}
yields
\begin{equation}
    0 \to f_*\ca O_{U'} \to f_*\tau_*\ca O_V \to f_*\bigoplus_{i=1}^{r+1} k_{q_i} \to R^1f_*\ca O_{U'} \to R^1f_*(\tau_*\ca O_V). 
\end{equation}
Now $R^1f_*(\tau_*\ca O_V) = R^1(f\circ\tau)_*\ca O_V$ since $\tau$ is affine, and this vanishes by considering each connected component separately. It remains to show that 
\begin{equation}
    f_*\tau_*\ca O_V \to f_*\bigoplus_{i=1}^{r+1} k_{q_i}
\end{equation}
is a surjection with kernel $\ca O_U$. This is immediate by writing out both sides to obtain
\begin{equation}
    k[x] \oplus k[y] \oplus \bigoplus_r k_q \to \bigoplus_{r+1} k_q, 
\end{equation}
with the map given by first evaluating at the point $q:x = y = 0$, and then applying the $r \times (r+1)$ matrix 
\begin{equation}
    \mat{1&-1&0&0&\dots & 0\\
    0 & 1 & -1 & 0 & \dots & 0\\
    0 & 0 & 1 & -1 & \dots & 0\\
    & & & \dots & & \\
    0 & \dots & 0 & 0 & 1 & -1 }. \qedhere
\end{equation}
\end{proof}

\subsection{Invariance VII: Shift of an index}
This is a new invariance, which was not present in \cite{Bae2020Pixtons-formula}. It changes the total degree, and so makes sense only for $\pcc$, not for $\DR$. % or $\rpcc$. 

Fix a positive integer $r$ and an index $j \in \{1, \dots, n\}$. Define $\ul a^j = (a_1, \dots, a_j + r, \dots, a_n)$. Consider the map
\begin{equation}
\phi_{\ca L, \ul a^j}\colon S \to \PIC_{g}^{d - \frak a -r}
\end{equation}
classifying $\ca L(-rp_j + \sum_i a_i p_i )$. 
\begin{lemma}\label{lem:index_shift}
\begin{equation}
\phi_{\ca L, \ul a}^*\pcc_{g}^{r, \bullet, d - \frak a}  = \left(1 + \frac{c_1(p_j^*\ca L)}{r}\right)\phi_{\ca L, \ul a^j}^*\pcc_{g}^{r, \bullet, d- \frak a - r}. 
%\phi_{\ca L, \ul a^j} ^* \pcc_{g}^{r, \bullet, d}  = \left(1 + \frac{c_1(p_j^*\ca L)}{r}\right)\phi_{\ca L, \ul a}^*\pcc_{g}^{r, \bullet, d+r}. 
\end{equation}
\end{lemma}
\begin{proof}
Let $\tilde S \to S$ be the moduli of $r$th roots of $\ca L(-\sum_i a_i p_i)$, for example defined as the fibre product 
\begin{equation}
\tilde S = S \times_{\phi_{\ca L, \ul a}, \PIC_g^{d - \frak a}, \epsilon} \PIC_g^{1/r, d - \frak a}. 
\end{equation}
Let $\pi \colon \tilde C\to \tilde S$ be the universal twisted curve, and $\ca F$ the universal $r$-th root of $\ca L(-\sum_i a_i p_i)$ on $\tilde C$. Define analogously the space $\tilde S'$ of $r$th roots of $\ca L(-rp_j + \sum_i a_i p_i )$, with universal twisted curve $\tilde C'$ and universal $r$th root $\ca F'$. Then there are canonical isomorphisms $\tilde S = \tilde S'$ and $\tilde C = \tilde C'$ given by the formula 
\begin{equation}
\ca F' = \ca F(-p_j). 
\end{equation}
This yields a short exact sequence
\begin{equation}
0 \to \ca F' \to \ca F \to \ca F|_p\to 0. 
\end{equation}
We find 
\begin{equation}
\pi_! \ca F = \pi_!\ca F' + \pi_!(\ca F|_{p_j}) = \pi_!\ca F' + {p_j}^*\ca F
\end{equation}
in $K(\tilde S)$. Multiplicativity of the total Chern class yields
\begin{equation}
c(\pi_! \ca F) =  c(\pi_!\ca F') c( p_j^*\ca F). % = p^*c(\pi_!\ca F') c( p_j^*\ca F). 
\end{equation}
Pushing forward along $\tilde S \to S$, we see
\begin{equation}
\phi_{\ca L, \ul a}^*\pcc_{g}^{r, \bullet, d - \frak a}  = \phi_{\ca L, \ul a^j}^*\pcc_{g}^{r, \bullet, d- \frak a - r}\epsilon_*(1 + c_1(p_j^*\ca F)). 
\end{equation}
Finally, since we do not allow twisting at markings, we obtain 
\begin{equation}
\epsilon_*(c_1(p_j^*\ca F))  = c_1(p_j^*\ca L)/r. \qedhere
\end{equation}
\end{proof}

\begin{remark}
Pulling back along the map $\Mbar_{g,n}\to \PIC_{g}^{2g-2 + n - \sum_i a_i}$ classifying $\omega_\log^s(-\sum_i a_i p_i)$, this lemma immediately recovers \cite[Theorem 4.1 (ii)]{GLN}. %; as such we are unconvinced by the assertion in [loc.cit.] that the geometric argument requires $s < 0$. 
\end{remark}

\section{Proof of \ref{introtheorem:DR}}
%\section{Proof of the Hodge-DR theorem}% of \ref{thm:uni_Hodge_DR,thm:polynomiality}}
\label{sec:proof_of_uni_Hodge_DR}

\subsection{The universal Hodge-DR theorem}\label{sec:proof_of_hodge_DR}

The proof of the `$u=0$' case of the Hodge-DR conjecture (see \ref{sec:hodge_DR}) in \cite{Bae2020Pixtons-formula} proceeds by first generalising the conjecture to a statement on the universal Jacobian, and then proving it by a reduction to results of \cite{Janda2018Double-ramifica}. Here we will proceed in a similar way for the general case. In this subsection we explain the `universal' generalisation of the Hodge-DR conjecture (for arbitrary $u$), and in \ref{sec:proof_of_Hodge_DR} we prove the generalised version by reduction to a result of \cite{Fan2019Higher-genus-re}. 

\begin{theorem}\label{thm:polynomiality}
After pulling back to any finite-type $k$-scheme mapping to $\PIC_{g}^0$, the classes $r^{2e - 2g + 1}\pcc_{g}^{r, e, d}$ form a polynomial in $r$ for $r$ sufficiently large. 
\end{theorem}
If we do not pull back to a finite-type scheme, then we get a power series in $r$, rather than a polynomial. 
\begin{proof}
Unless $d=0$ then this class is constant for $r$ sufficiently large; it is $1$ if $e=0$ and $0$ otherwise. As such, we reduce to the case where $d=0$. The result is then proven in \ref{sec:proof_of_Hodge_DR}. 
\end{proof}

%\david{This $\rpcc$ thing seems like unnecessary extra notation to introduce? Do we even need the $\pcc$? Maybe it's handy in the invariance section, but for stating main theorems...?}
%\begin{definition}
%Given a non-negative integer $u$, we write $\rpcc_{g}^{e}[r^u]$ for the coefficient of $r^u$ in the polynomial given by $r^{2e - 2g + 1}\pcc_{g}^{r, e, 0}$
%for $r$ sufficiently large. 
%\end{definition}

\begin{definition}
Given a non-negative integer $u$, we write $\rpcc_{g}^{e}[r^u]$ for the coefficient of $r^u$ in the polynomial given by $r^{2e - 2g + 1}\pcc_{g}^{r, e, 0}$ for $r$ sufficiently large. 
\end{definition}

We can now state precisely the universal Hodge-DR theorem: 
\begin{theorem}\label{thm:uni_Hodge_DR}
%Fix $g$, $n$, $\ul a = (a_1, \dots, a_n)$ with $\sum_i a_i = d$. Let $u$ be a non-negative integer. Define
%\begin{equation}
%\phi\colon \PIC_{g,n,d} \to \Pic_g^d; \;\;\; (C/S, p_1, \dots, p_n, \ca L) \mapsto (C/S, \ca L(-\sum_i a_i p_i)). 
%\end{equation}
For every non-negative integer $u$ we have
\begin{equation}
\rpcc_{g}^{g+u}[r^u] = \DR_{g}[u] \in \Chow^{g + u}(\PIC_{g}^0). 
\end{equation}
\end{theorem}
In fact both classes make sense outside total degree zero, and both vanish there for trivial reasons, so the reader who prefers can read the both theorem as a statement over the whole of $\PIC_g$. 

%\david{If we don't want to introduce the $\rpcc$ notation we instead write 
%\begin{equation}
%(r^{2u + 1}\pcc_{g}^{r, g+u, 0})[r^u] = \DR_{g}[u] \in \Chow^{g + u}(\PIC_{g}^0). 
%\end{equation} But it's a pain in the proof later when we want an $\rpcc^\bullet$. So I think leave as-is for now (though in the intro version of the theorem don't use it). }

Pulling back \ref{thm:uni_Hodge_DR} along the map $\Mbar_{g,n} \to \PIC_g^0$ given by $\omega_\log^s (-\sum_i a_i p_i)$ immediately yields the following. 
\begin{corollary}
The Hodge-DR conjecture is true. 
\end{corollary}
% \begin{proof}
% Pull back along map $\Mbar_{g,n} \to \PIC_g^0$ given by $\omega_\log^s (-\sum_i a_i p_i)$. But there is more to say here! 
% \end{proof}

\subsection{The proof of \ref{thm:polynomiality,thm:uni_Hodge_DR}}
\label{sec:proof_of_Hodge_DR}
Using the machinery we have set up so far, the proof now mirrors closely that of the main theorem of \cite{Bae2020Pixtons-formula}. We avoid repeating the steps in full detail, but instead focus on the points where our argument is significantly different. 

\Cref{thm:uni_Hodge_DR} is an equality in the Chow cohomology of the Artin stack $\PIC_g^0$. Our first move is to re-write it in a more concrete form, as a statement about arbitrary families of prestable curves over DM stacks. 

\begin{theorem}\label{thm:Hodge_DR_Expanded}
Let $(C/S, p_1, \dots, p_n)$ be a family of prestable curves with $S$ a quasi-compact DM stack of finite type over $k$, let $\ca L_0$ be a line bundle on $C$ of relative degree $d$, and let $a_1, \dots, a_n$ be integers summing to $d$. Define $\ca L = \ca L_0(-\sum_i a_i p_i)$, and let
\begin{equation}
\phi\colon S \to \PIC_g^{1/r, 0}
\end{equation}
be the classifying map of $\ca L$. Then $\phi^*\Xi_g^{g + u}$ is a polynomial in $r$ for $r$ sufficiently large, and we have
\begin{equation}
\phi^*\DR_g[u] = \phi^*\Xi_g^{g + u}[r^u] \in \Chow^{g+u}(S),
\end{equation}
where as always the notation $[r^u]$ means to take the coefficient of $r^u$ in the polynomial obtained by making sufficiently large $r$. 
\end{theorem}
This theorem is easily seen to be equivalent to \Cref{thm:uni_Hodge_DR}. 
\begin{proof}[Proof of \ref{thm:Hodge_DR_Expanded}] The theorem is a statement about arbitrary families, but the proof begins by showing that it is in fact sufficient to treat rather special families. 

\textbf{Step 1:} By \cite[Lemma 43]{Bae2020Pixtons-formula}, we may assume that $C/S$ has \emph{enough sections}: for every geometric fibre of $C/S$, the complement of the union of irreducible components carrying markings is a disjoint union of trees on which $\ca L$ is trivial. 

\textbf{Step 2:} By repeated replacing $a_i$ by $a_i + 1$ and $\ca L$ by $\ca L(p_i)$ (here we use Invariances II and III), we may assume that $\ca L$ is \emph{relatively sufficiently positive}: $\ca L$ is relatively base-point free, and satisfies $R^1\pi_*\ca L = 0$. 

\textbf{Step 3:} Copying the proof of \cite[Lemma 41]{Bae2020Pixtons-formula}, we reduce to the case treated in Step 4 (here we use Invariance VI). 

\textbf{Step 4:}
Fix a positive integer $l$, and write
\begin{equation}
S = \Mbar_{g,n}(\bb P^l, dL)' \subseteq \Mbar_{g,n}(\bb P^l, dL) 
\end{equation}
for the open locus of stable maps $f\colon C \to \bb P^l$ of degree $d$ such that, on each geometric fibre,  $H^1(C, f^*\ca O(1)) = 0$. The fundamental and virtual fundamental classes of $\Mbar_{g,n}(\bb P^l, dL)'$ coincide by an easy calculation (\cite[Lemma 39]{Bae2020Pixtons-formula}). 

Let $C/S$ be the universal curve, let $\ca L_0$ be the pullback of $\ca O(1)$ from $\bb P^l$, and define
\begin{equation}
\phi_\ca L\colon \Mbar_{g,n}(\bb P^l, dL)' \to \PIC_{g}^0; \;\;\; (C/S, f\colon C \to \bb P^l) \mapsto f^*\ca O(1)(-\sum_i a_i p_i), 
\end{equation}
and we will prove \ref{thm:uni_Hodge_DR} after pulling back along $\phi_\ca L$. The precise statement is \ref{lem:apply_FWY}. 
\end{proof}
\begin{lemma}\label{lem:apply_FWY}
Let $C/S$, $\ca L$ be as in Step 4 of the proof of \ref{thm:Hodge_DR_Expanded}. Then $\phi_\ca L^*\rpcc_{g}^{g+u}$ is polynomial in $r$ for $r$ sufficiently large, and we have 
\begin{equation}
\phi_\ca L^*\DR_{g}[u] = \phi_\ca L^*\rpcc_{g}^{g+u}[r^u]. 
\end{equation}
\end{lemma}
The proof below follows closely that of \cite[Theorem 6.4]{Chen2022A-tale-of-two-m}, where we reduce to a result of \cite{Fan2019Higher-genus-re}. 
\begin{proof}
As $\Mbar_{g,n}(\bb P^l, dL)'$ is smooth and has fundamental class equal to its virtual fundamental class, it is enough to check this equality after capping with the (virtual) fundamental class.

Write $\Mbar_{g,n,\ul a, \beta}^\sim(\bb P^l, \ca O (1))$ for the space of stable rubber maps of degree $\beta$ to the projectivised line bundle $\ca O(1)$ on $\bb P^l$, relative to the divisors given by the zero and infinity sections of the projectivised bundle, with contact orders given by $\ul a$. Write $\Psi_\infty$ for the class of the cotangent line of the marked divisor at infinity. We write 
\begin{equation}
p\colon \Mbar_{g,n, \ul a, \beta}^\sim(\bb P^l, \ca O (1)) \to \Mbar_{g, n, \beta}(\bb P^l)
\end{equation}
for the projection, and we denote by $\Mbar_{g,n,\ul a, \beta}^\sim(\bb P^l, \ca O (1))'$ the preimage of $\Mbar_{g, n, \beta}(\bb P^l)'$. A slight generalisation of \cite[Lemma 6.2]{Chen2022A-tale-of-two-m} then yields
\begin{equation}
\phi_\ca L^*\DR_{g}[u] = p_*([\Mbar_{g,n, \ul a, \beta}^\sim(\bb P^l, \ca O (1))']^\vir \cdot \Psi_\infty^u). 
\end{equation}

Write $I_\infty \subseteq \{1, \dots, n\}$ for the set of those $i$ such that $\ul a_i < 0$. Define a vector $\ul a(r)$ by 
\begin{equation}
\ul a(r)_i = \begin{cases}a_i & \text{if } a_i \ge 0\\
r + a_i & \text{else}. 
\end{cases}
\end{equation}
Define
\begin{equation}
\begin{split}
\phi_r\colon\Mbar_{g,n}(\bb P^l, dL)' & \to \PIC_{g}^{1/r, -r\#I_\infty}\\
(C/S, p_1, \dots, p_n, f) & \mapsto f^*\ca O(1)(-\sum_i \ul a(r)_i p_i). 
\end{split}
\end{equation}
Write %\ale{I suspect some Ch are $\Xi$ or at least I remember it was the case. Or at least I find it hard to follow all these notations Ch $\Xi$ and $\Omega$}\david{I agree it is painful, I did make some effort to remove them, but then it got worse (many formulae that did not fit on a single line). }
\begin{equation}
\on{Ch}_{g, \ul a(r)}^{r, \bullet} = \sum_{d\ge 0}r^{2g - 2d + 1}\phi_r^*\pcc_{g}^{r, d} \in \Chow^*(\Mbar_{g,n}(\bb P^l, dL)')[r]. 
\end{equation}
Writing $\sf{ev}_i\colon \Mbar_{g,n}(\bb P^l, dL)'  \to \bb P^l$ for the $i$th evaluation map and setting $\nu_i = c_1(\sf{ev}_i^*\ca O(1))$, \ref{lem:index_shift} (Invariance VII) yields 
\begin{equation}\label{eq:Ch_shift}
\on{Ch}_{g, \ul a(r)}^{r, \bullet} = \left(\phi_\ca L^*\rpcc_g^{r, \bullet}\right)\cdot\prod_{i \in I_\infty}(1 + r (a_i \psi_i + \nu_i));
\end{equation}
note that the $r$ has moved to the numerator relative to the expression in \ref{lem:index_shift}, because of the factors of $r$ relating $\rpcc$ and $\pcc$. 
Applying \cite[Corollary 4.4]{Fan2019Higher-genus-re} yields for $r$ sufficiently large that this expression is polynomial in $r$, and that 
\begin{equation*}
\begin{split}
p_*([\Mbar_{g,n,\ul a, \beta}^\sim & (\bb P^l, \ca O (1))]^\vir \cdot \Psi_\infty^u) \\
& = \sum_{\ul e \in \bb Z_{\ge 0}^{I_\infty}} \prod_{i \in I_\infty} (-a_i \psi_i - \nu_i)^{e_i} \cdot  \left(\phi_\ca L^*\on{Ch}_{g, \ul a(r)}^{ r, u + g - \abs{\ul e}}[r^{u - \abs{\ul e}}]\right)\\
& = \left[  \sum_{\ul e \in \bb Z_{\ge 0}^{I_\infty}} \prod_{i \in I_\infty} (-r(a_i \psi_i + \nu_i))^{e_i}  \phi_\ca L^*\on{Ch}_{g, \ul a(r)}^{r, \bullet} \right]_{\text{codim} g + u}[r^u] \\
& = \left[  \prod_{i \in I_\infty} \frac{1}{1 + r(a_i \psi_i + \nu_i)} \phi_\ca L^* \on{Ch}_{g, \ul a(r)}^{r, \bullet} \right]_{\text{codim} g + u}[r^u] \\
& = \left[  \phi_\ca L^*\rpcc_g^{r, \bullet}\right]_{\text{codim} g + u} [r^u]\\
\end{split}
\end{equation*}
where for the last equality we apply \ref{eq:Ch_shift}. 
\end{proof}

\section{Proof of \ref{introtheorem:logDR}}\label{sec:proof_of_thm_logDR}

\begin{definition}
Let $C/S$ be a log curve. A \emph{quasi-stable model} for $C/S$ is a map of log curves $f\colon C'\to C$ such that the geometric fibres of $f$ are points or smooth rational curves meeting the remainder of the curve at 2 points. If $C' \to C$ is a quasi-stable model and $\ca F$ is a line bundle on $C'$, we say $\ca F$ is \emph{admissible} if it has degree $1$ on every exceptional curve of $C' \to C$. 
\end{definition}

\begin{definition}
Let $C/S$ be a log curve. The \emph{quasi-stable Jacobian} $\Jac^Q_{C/S}$ represents the strict \'etale sheafification of the functor $\cat{LogSch}_S^\op \to \cat{Set}$ sending a log scheme $T \to S$ to the set of isomorphism classes $(C' \to C_T, \ca F)$ where $C' \to C_T$ is a quasi-stable model and $\ca F$ is an admissible line bundle on $C'$ of total degree 0. 
\end{definition}

In general the quasi-stable jacobian is highly non-separated, but it is designed to contain compactifications of the classical jacobian. 
\begin{definition}\label{def:comp_jac}
Let $C/S$ be a log curve. A \emph{compactified Jacobian} for $C/S$ is an open subspace of $\ca J \subseteq \Jac^Q_{C/S}$ such that
\begin{enumerate}
\item $\ca J$ is proper over $S$;
\item $\ca J$ contains the subfunctor $\Jac^{\ul 0}_{C/S}$ of line bundles of multidegree $\ul 0$; 
\item\label{item:comp_jac_sec} For every geometric point $(C' \to C_t, \ca F)$ of $\ca J$, either $\ca F$ is trivial or $h^0(C', \ca F) = 0$. %and every twisted $r$th root $\ca L^{1/r}$ of $\ca L$, either $\ca L^{1/r}$ is trivial or $h^0(C', \ca L^{1/r}) = 0$. 
\end{enumerate}
\end{definition}

If $C/S$ is stable or admits at least one section through the smooth locus then such a compactified Jacobian exists, by \cite{Kass2017The-stability-s,Esteves2001Compactifying-t}. The first two conditions are common requirements for compactified Jacobians. The third condition also holds in many standard examples, in particular 
\begin{enumerate}
\item
Suppose that $\ca J$ is associated to a stability condition $\theta$ in the sense of \cite{Kass2017The-stability-s}. If $\theta$ is nondegenerate then condition (1) is satisfied, if $\theta$ is small (i.e. the trivial bundle is stable) then condition (2) holds, and in this situation condition (3) also holds by \cite[Lemma 8]{Holmes2017Jacobian-extens}, based on a lemma of Dudin \cite{Dudin2015Compactified-un}. This case is used in \cite{Holmes2022Logarithmic-double}. 
\item Suppose that $C/S$ admits a section $x$ through the smooth locus, let $\theta$ be a small stability condition (not assumed nondegenerate, so we can take the zero stability condition), and $\ca J$ be the quasi-stable Jacobian associated to $(\theta, x)$ in the sense of \cite{Esteves2001Compactifying-t} (the argument is essentially the same as in the first case). %. %\david{This is easy to prove, and follows a similar pattern to the one above. The proof is short, but ti requires writing out a whole bunch of stability condition machinery to check, and I prefer not to do that since we don't need it for anything else. Can we just say the argument is basically the same and leave it there? Ale says it sounds reasonable, maybe can give a hint of proof, or say it is similar to elsewhere. }
\end{enumerate}

%\david{Say something about such a $\ca J$ existing as long as curves are stable or there is at least one section? }

From now on we assume that a compactified jacobian $\ca J$ exists, and we choose one. We write $C' \to C\times_S \ca J$ for the universal quasi-stable curve and $\ca F$ for the universal quasi-stable line bundle. 

For later use, we show that condition (3) for $\ca F$ implies that the same holds for all $r$th roots: 
\begin{lemma}\label{lem:cohomology_cond_hold_for_root}
Let $(C' \to C_t, \ca F)$ be a geometric point of $\ca J$, let $r$ be a positive integer, and let $\ca F^{1/r}$ be an $r$th root of $\ca F$ on a twisted curve $\ca C \to C'$. Then either $\ca F^{1/r}$ is trivial or $h^0(\ca C, \ca F^{1/r}) = 0$. 
\end{lemma}
\begin{proof}
Suppose that $h^0(\ca C, \ca F^{1/r}) \neq 0$, so there exists a non-zero global section $s$ of $\ca F^{1/r}$. Then $s^r$ is a non-zero global section of $\ca F$, hence $\ca F$ is trivial by \ref{def:comp_jac} (3), hence $s^r$ is nowhere $0$. This implies that $s$ is nowhere 0, hence it is a trivialising section for $\ca F^{1/r}$. 
\end{proof}

We write $\epsilon\colon \ca J^{1/r}\to \ca J$ for the stack of twisted $r$th roots of the universal line bundle $\ca F$ on $C'$. Let $E^r$ be the closed locus in $\ca J^{1/r}$ where the universal root is trivial. This is a closed regularly-embedded subscheme of codimension $g$ (since it is contained in $\Jac^{\ul 0}_{C/S}$). 

\begin{lemma}\label{lem:Er_push_E}
We have $\epsilon_*[E^r] = \frac 1 r [E]$
\end{lemma}
\begin{proof}
$E^r$ is contained in the locus where the map $C' \to C\times_S \ca J$ is an isomorphism and the line bundle $\ca F$ has multidegree 0. On this locus the map $\epsilon$ is \'etale, and the induced map $E^r \to E$ is a $\mu_r$-gerbe, since a bundle is trivial if and only if exactly one of its $r$th roots is trivial (the gerbe structure comes from the action of $\mu_r$ on the roots, and contributes the $\frac 1 r$). 
\end{proof}

\begin{theorem}\label{thm:logDR_chern}
Fix a non-negative integer $u$. In $\Chow^g(\ca J^{1/r})$ we have the equality 
\begin{equation}
[E^r] \cdot  c_1(p^*\ca F^{1/r})^u = c_{g+u}(-R\pi_*\ca F^{1/r})
\end{equation}
where $p\colon S \to C$ is any section. 
\end{theorem}
\begin{proof}
We first treat the case $u=0$. To simplify notation in this proof we write $\ca F$ in place of $\ca F^{1/r}$ for the universal line bundle on the universal twisted curve over $\ca J^{1/r}$. Let $D$ be a sufficiently ample divisor on $C'$, then the exact sequence
\begin{equation}
0 \to \ca F \to \ca F(D) \to \ca F(D)|_D \to 0
\end{equation}
yields an exact sequence 
\begin{equation}
0 \to \pi_*\ca F \to \pi_*\ca F(D) \stackrel{\sigma}{\to} \pi_*\ca F(D)|_D \to R^1\pi_*\ca F \to 0, 
\end{equation}
in which the middle two terms are vector bundles. By \ref{lem:cohomology_cond_hold_for_root} we see that $E_r$ is precisely the degeneracy locus of the map $\sigma$, which is computed by the Thom--Porteous formula \cite[14.4]{Fulton1984Intersection-th} as $c_g(-R\pi_*\ca F)$. 

We treat now the case $u \ge 1$. The sequence 
\begin{equation}
0 \to \ca F(-p) \to \ca F \to \ca F_p \to 0
\end{equation}
yields %\david{Replace the next few $\pi_!$ with $R^\bullet \pi_*$? Here we are really computing in K-theory, so perhaps better left as it is? }
\begin{equation}
\pi_! \ca F = \pi_!\ca F(-p) + p^*\ca F. 
\end{equation}
and $-\pi_! \ca F (-p)$ is a vector bundle of rank $g$ (since our condition (3) on $\ca J$ implies that $\pi_*\ca F(-p) = 0$ universally). 
Now
\begin{equation}
c( -R\pi_*\ca F(-p) ) = c(-R\pi_*\ca F) (1+ c_1(p^*\ca F))
\end{equation}
and so for every positive integer $d$, taking the part in degree $g+d$ yields 
\begin{equation}
0 = c_{g+d}( -R\pi_*\ca F(-p) ) = c_{g+d}(-R\pi_*\ca F)  + c_1(p^*\ca F))c_{g+d-1}(-R\pi_*\ca F)
\end{equation}
and hence
\begin{equation}
c_{g+d}(-R\pi_*\ca F) = c_1(p^*\ca F^\vee))c_{g+d-1}(-R\pi_*\ca F)
\end{equation}
By induction on $d \ge 1$ we see that 
\begin{equation}
c_{g+d}(-R\pi_*\ca F) = c_1(p*\ca F^\vee))^d c_{g}(-R\pi_*\ca F), 
\end{equation}
and taking $d=u$ yields the required result. 
\end{proof}

It remains to put together \ref{lem:Er_push_E} and \ref{thm:logDR_chern} to deduce \ref{introtheorem:logDR}. 

%Let $C/S$ be a log curve and assume that a fine compactified Jacobian $\ca J$ exists for $C/S$. Write $C' \to C_\ca J$ for the universal quasi-stable model, and $\ca F$ for the universal line bundle. 

Let $\ca L$ be a line bundle on $C$ of relative degree $0$. We define $\tilde S = \tilde S_{\ca J, \ca L}$ to represent the functor sending a log scheme $T/S$ to the set of tuples
\begin{equation}
(f\colon C' \to C_T, \beta)
\end{equation}
where $f\colon C' \to C_T$ is a quasi-stable model and $\beta$ is a reduced PL function, such that $(f\colon C' \to C_T, f^*\ca L(\beta))$ is a $T$-point of $\ca J$. There is a map $\tilde S \to \ca J$ sending $ (f\colon C' \to C_T, \beta)$ to $(C' \to C, f^*\ca L(\beta)))$. 

\begin{remark}
The space $\tilde S$ is a log alteration\footnote{A generalisation of log modifications which allow root stacks along the boundary. } of the fibre product over the \emph{logarithmic Picard scheme} $\LogPic_{C/S}$ of Molcho and Wise \cite{Molcho2018The-logarithmic}:
\begin{equation}
S \times_{\LogPic_{C/S}} \ca J, 
\end{equation}
where $S \to \LogPic$ is the classifying map for (the log line bundle associated to) $\ca L$, and $\ca J \to \LogPic$ sends a pair $(C' \to C_T, \ca F)$ to the log line bundle on $C_T$ associated to $\ca F$. The alteration serves to arrange the reducedness condition on the PL function (see \ref{sec:PL_functions}). 
\end{remark}

\begin{remark}
The morphism $\tilde S_{\ca J, \ca L} \to S$ is a proper log monomorphism. If $S$ is log regular then it is birational (in general it is an isomorphism over the locus $U \subseteq S$ over which $C$ is smooth). 
\end{remark}
%\begin{proof}\david{to expand}
%It is a monoidal alteration of the fibre product $S \times_{\LogPic_{C/S}} \ca J$, and $\ca J$ is proper over logPic. 
%\end{proof}

%Fix a positive integer $r$. We write $\tilde S_{\ca J, \ca L}^{1/r}$ for the space of twisted $r$th roots of the line bundle over $\tilde S_{\ca J, \ca L}$, with $\ca C$ the corresponding twisted curve. 

\begin{lemma}\label{lem:E_pullback_logDR}
The schematic pullback of the zero section $E$ from $\ca J$ to $\tilde S$, equipped with its strict log structure, is equal as a log scheme to the logarithmic double ramification locus $\logDRL(\ca L)$. The pullback of the fundamental class $[E]$ is equal to the push-forward of the virtual fundamental class of $\logDRL(\ca L)$ to $\tilde S$. 
\end{lemma}
\begin{proof}
The first assertion is easily verified, as they represent the same functors on log schemes. For the second assertion, note that $E$ is contained in $\Jac^{\ul 0}_{C/S} \subseteq \ca J$, and the preimage of $\Jac^{\ul 0}_{C/S}$ in $\tilde S$ is exactly the space $\tilde S^\ca L$ from \ref{eq:StildeL}. The virtual fundamental class of $\logDRL(\ca L)$ is defined as the pullback of the fundamental class of the zero section. 
\end{proof}

Combining \ref{lem:Er_push_E}, \ref{thm:logDR_chern}, \ref{lem:E_pullback_logDR}, and the fact that $rc_1(p^*\ca F^{1/r}) = c_1(p^*\ca F)$, we deduce that for every positive integer $r$ the equality
\begin{equation}\label{eq:logDR_formula}
\logDR(\ca L) \cdot  c_1(p^*\ca F)^u = r^{u+1} \epsilon_*c_{g+u}(-R\pi_*\ca F^{1/r})
\end{equation}
holds in the Chow ring of $\tilde S$. We have proven \ref{introtheorem:logDR}, where the line bundle $\tilde{\ca L}$ is defined as the pullback of $\ca F$. 

%\david{I think we combine this with Theorem D in case $u=0$ to recover main result of HMPPS. I think that is stated in the intro, but maybe emphasise again here. }

\section{Riemann--Roch calculations}%Proof of \ref{introtheorem:ch_formula}}

\label{sec:main_computation}

\subsection{The twisted curve and the desingularised curve.}

We write $\pi\colon \ca C \to \PIC^{1/r}$ for the universal $r$-twisted curve,  $\ca L^{1/r}$ for the universal $r$th root on $\ca C$, and $\mathfrak r\colon \ca C \to C$ for the relative coarse moduli space of $\ca C$ over $\PIC^{1/r}$. %The $r$th power $\ca L$ of $\ca L^{1/r}$ descends along $\mathfrak r$. 

We write $\Sing$ for the singular locus of $\pi$ and 
$\SSing\to \Sing$ for the stack classifying the branches of the nodes; 
we write $\sigma\colon \SSing \to \SSing$ for 
the deck involution. 
The locus $\SSing$ decomposes as  \begin{equation}\label{eq:SingDecomp}
 i=\bigsqcup_{a=0}^{r-1}i^a\colon\bigsqcup_{a=0}^{r-1}
 \SSing_a=\SSing\longrightarrow \Sing\hookrightarrow \ca C,
 \end{equation}
 because $\ca L^{1/r}$ and the  bundle 
of tangent lines along the privileged branch in $\SSing$
are $\pmb\mu_r$-linearized: the first 
$\pmb\mu_r$-representation is an $a$th power of 
the second for a suitable locally constant index $a\in \{0,\dots, r-1\}$.
We set $j^a\coloneqq\pi\circ i^a$. 
The first Chern class of the cotangent line  at the chosen branch 
equals $\frac1r$ times the first Chern class $\psi$  of 
 the  line bundle cotangent to the coarse branch. 
We write  $\psi'\coloneqq \sigma^*\psi$ for 
the classes attached to the second branch.
Finally, the first Chern class of the universal line bundle $\ca L$ on $\ca C$
will be denoted by $\xi$ and the pullback to $\SSing_a$ will be denoted by $\xi_a$.

At each point $p$ of $\Sing$ consider the Cartier divisor $\Delta$ of $\PIC^{1/r}$ where 
the node persists. The local picture of $\ca C$ at $p$
can be written as the quotient stack 
$[(xy=t)/\pmb \mu_r]$ where $t$ the is the local parameter of $\Delta$ and 
$\zeta\in \pmb \mu_r$ operates as $(x,y)\mapsto(\zeta x,\zeta^{-1}y)$. 
The local picture of $C$ at $p$ can be written as 
$(xy=t^r)$ and, 
by blowing up iteratively, 
we get 
the desingularisation 
$\widetilde \pi\colon \widetilde C\to \PIC^{1/r}$.
We have 
$$\ca C\xrightarrow{ \ \ \mathfrak r\  \ } C\xleftarrow{\ \ \kappa\  \ }\widetilde C.$$
Blowing $\widetilde C$ down  to $C$ yields the contraction $\kappa$ of chains of $r-1$ projective lines on all nodes of $C$. After base change via $i\colon \SSing_a\to C$ 
the exceptional divisor $\on{Exc}(\kappa)$  can be written as a family of chains of $r-1$ projective lines $\cup_{i=1}^{r-1} E^a_i$ over $\SSing$ 
with the $\bb P^1$-bundles $E^a_i$ labelled from $1$ to $r-1$ starting from the base change of the projective line 
in $\widetilde C$ meeting the 
privileged branch corresponding to the base point in $\SSing$. 
Abusing notation, we  refer by $E^a_i$ to the image 
of $E^a_i$ in $\widetilde C$; in this way we have $E^a_i=E^b_{j}$ for $i+j, a+b\equiv 0\ (r)$. 
The fibre of $\widetilde C$ over $\Delta$ is the divisor $\widetilde C_\Delta=C_\Delta+ \on{Exc}(\kappa),$ %
%=C_\Delta + \sum_{i=1}^{\lfloor r/2\rfloor} E_i,$$
where $C_\Delta$ is the strict transform of $C$ in 
$\widetilde C$. 
We also write $\tilde \imath\colon \SSing^\sim\to \Sing^\sim$
for the double cover of the singular locus $\Sing^\sim\in \widetilde C$ and $\widetilde\psi$
 for the first Chern classes
of the cotangent line bundles at the privileged 
branch. We write $\widetilde{\psi}'$ when we switch to 
the other branch. 

% ; then each choice of a privileged branch 
% on a node $n\in \widetilde C$ of the chain induces a choice of a branch of the node $\kappa (n)\in C$ and 
% there is an $r$-to-$1$ 
% morphism $\SSing^\sim\to \Sing^\sim$ compatible with the
% $r$-to-$1$ restriction of $\tilde \pi$ mapping  $\Sing^\sim$ to $\Sing$. The points of $\SSing^\sim$
% over $\SSing$ are numbered by the index 
% $k\in \{1,\dots, r\}$ counting the nodes of 
% the exceptional divisor lying on the same 
% side of the prescribed branch. 
% In this way $\SSing^\sim$ decomposes into $r$ connected components 
% $\SSing^1,$ \dots, $\SSing^r$ 
% each one equipped with the classes $\psi_k$ and $\psi'_k$;
% furthermore, the decomposition \ref{eq:SingDecomp} of $\SSing$ 
% yields the refined decomposition 
% of $\SSing^\sim$ and $\widetilde i$
% $$\SSing^\sim=\bigsqcup_{a=0}^{r-1}\bigsqcup_{k=1}^r \SSing^\sim_{a,k}, \qquad  \widetilde{i}=\bigsqcup_{a,k} \left (i^\sim_{a,k}\colon 
%  \SSing_{a,k}^\sim\to \Sing^\sim\right).$$
% A node in $\Sing^\sim$ lies in the image of exactly two 
% morphisms: $i^\sim_{a,k}$ and $i^\sim_{r-a,r-k+1}$.

The line bundle $\ca L$ on $\ca C$ 
is the pullback of a line bundle  via $\mathfrak r$. We write $\widetilde L$
for the pullback of the same line bundle on $\widetilde C$. 
On $\widetilde C$, there is a line bundle $\widetilde L^{1/r}$
whose $r$th power is isomorphic to $\widetilde L(-D)$ for $D=\sum_{a=0}^{\lfloor r/2\rfloor } D_a$  with 
\begin{equation}\label{eq:chains}  
 D_a= \sum_{i=1}^{r-a}iaE^a_i+\sum_{i={r-a+1}}^{r-1}i(r-a)E^a_{r-i},
\end{equation}
for $a\neq 0, r/2$ and $D_a=\sum_{i=1}^{r-a}iaE^a_i$ otherwise,
and whose push-forward on $C$ satisfies 
$$\mathfrak r_*\ca L^{1/r}\cong \kappa_*\widetilde L^{1/r}$$
(this setup is recalled in 
\cite[Lem.~2.2.5]{Chiodo2008Towards-an-enum}, see also 
Fig.~3 therein picturing the divisor $D_a$). 
%
%
%
%where $\sum_{i=0}^{r-1}E_i$ 
%is the exceptional divisor of $\kappa$, see \cite{Chiodo2008Towards-an-enum}.
%\begin{figure}[h]
%\begin{picture}(200,60)(290,0)
%  \qbezier(220,5)(265,50)(265,50)
%\put(252,31){\rotatebox{315}{\footnotesize$2(r-a)$}}
%  \qbezier(250,50)(295,5)(295,5)
%\put(225,21){\rotatebox{45}{\footnotesize$r-a$}}
%%\put(224,21){\rotatebox{45}{\footnotesize$3(r-a)$}}
%\put(385,20){\rotatebox{45}{\footnotesize $(r-a)a$}}
%\put(338,48){\rotatebox{315}{\footnotesize $\!(r-a)(a-1)$}}
%\put(429,45){\rotatebox{315}{\footnotesize $(r-a-1)a$}}
%\put(495,25){\rotatebox{45}{\footnotesize $2a$}}
%\put(542,30){\rotatebox{315}{\footnotesize$a$}}
%%  \qbezier(270,5)(315,50)(325,50)
%%  \qbezier(250,50)(295,5)(295,5)
%  \put(310,30){$...$}
%  \qbezier(460,5)(415,50)(415,50)
%  \qbezier(430,50)(385,5)(385,5)
%  \qbezier(400,5)(355,50)(355,50)
%  \put(472,30){$...$}
% % \qbezier(510,5)(465,50)(465,50)
% % \qbezier(480,50)(435,5)(435,5)
%  \qbezier(530,50)(485,5)(485,5)
%  \qbezier(560,5)(515,50)(515,50)
% \end{picture}
%\caption{The $a$-chain divisor $D_a$ lying above 
%over the coarsening of the node of 
%multiplicities $\frac1r(r-a,a)$.}
% \end{figure}
%The direct image 
%on the base scheme $S$ can be equivalently 
%computed stack-theoretically as $ R\pi_*\mathcal L^{1/r}$
%or scheme-theoretically as $R\widetilde \pi_*\widetilde L^{1/r}$. 
%therefore, we get  
%\begin{equation}\label{eq:schemeGRR}\on{ch}(R\pi_* \mathcal L^{1/r})=\on{ch}(R\pi_* \widetilde L^{1/r})=
%\widetilde \pi_* \left (
%\on{ch}(\widetilde L^{1/r})\on{td}(\Omega^\vee_{\widetilde \pi})\right).\end{equation}
%

\subsection{Proof of \ref{introtheorem:ch_formula}}
We express $\on{ch}(R\pi_*\ca L^{1/r})$ as 
$\on{ch}(R\widetilde \pi_*\widetilde L^{1/r})$.  GRR yields\footnote{Here we apply GRR to a representable morphism from $\tilde C$ to the Artin stack $\PIC^{1/r}$, whereas the form of the theorem we use is only available for morphisms of schemes. This is justified because we work always in Chow cohomology with test objects being separated schemes of finite type; as such, the assertion that GRR holds for $\tilde C \to \PIC^{1/r}$ is formally identical to the assertion that GRR holds for the base-change $\tilde C\times_{\PIC^{1/r}}S \to S$ for $S$ a separated scheme of finite type, and the observation that both sides of the formula commute with base-change. The same remark applies later in the proof at \ref{eq:mumfordlocal} where we use GRR again to compute some terms of the right hand side.  }
$$\on{ch}\left(R\widetilde \pi_*\widetilde L^{1/r}\right) =    \widetilde\pi_*\left(\exp(c_1(\widetilde L)/r)\exp(-D/r)\on{td}\Omega_{\widetilde \pi}^\vee\right)$$
with $$\on{td}^\vee(L)\coloneqq\on{td}(L^\vee)=\frac{c_1(L)}{e^{c_1(L)}-1}=\sum_{n\ge 0}
\frac{B_n}{n!}c_1(L)^n.$$
We re-write as $\widetilde\pi_*\left(({{\pmb 1}+U}) ({\pmb 1}+E)\on{td}\Omega_{\tilde \pi}^\vee\right)$,
with $$\pmb 1 + U=\exp(c_1(\widetilde L)/r)=\kappa^* \left(\exp(c_1(\ca L)/r)\right)=\kappa^* \left(\exp(\xi/r)\right)$$ and $\pmb 1+E=\exp(-D/r)$.
% By \cite[Lemma 29.32.10]{stacks-project}, 
The formation of $\Omega_{\tilde\pi}$ is compatible with
the base change 
\begin{equation}
 \begin{tikzcd}[row sep=small]
  \widetilde C \arrow{d}[swap]{\tilde \pi\,}  
  \arrow[r, "\varphi"] & C_{\mathfrak M}\arrow[d, "\,\mathfrak p"] \\
  S\arrow[r, "\mathfrak f"]& \mathfrak M
\end{tikzcd}
\end{equation}
from
the universal 
curve $C_{\mathfrak M}$ over the moduli stack  
of prestable curves 
$\mathfrak M$.
The exact sequence 
\begin{equation}\label{eq:Omegaomega} 
0\to \Omega_{\mathfrak p} \to \omega_{\mathfrak p}\to \omega_{\mathfrak p}|_{\Sing_{\mathfrak p}}=\ca O_{\Sing_{\mathfrak p}}\to 0
\end{equation}
% \begin{equation}\label{eq:Omegaomega} 
% 0\to \Omega_{\tilde C/\PIC^{1/r}} \to \omega_{\tilde C/\PIC^{1/r}}\to \omega_{\tilde C/\PIC^{1/r}}|_{\Sing^\sim}=\ca O_{\Sing^\sim}\to 0
% \end{equation}
yields
 %see Mumford 
%\cite{Mumford1983Towards-an-enum} and
%(this Mumford 
%\cite{Mumford1983Towards-an-enum} argument adapted by Tseng \cite[Lemma 7.3.3]{tseng2010orbifold}
%for twisted curves)
\begin{multline*}\label{eq:toddsplit}
\on{td}^\vee\Omega_{\tilde\pi}=
\on{td}^\vee(\varphi^*\Omega_\mathfrak p)=
\on{td}^\vee(\varphi^*\omega_\frak p - \varphi^*\ca O_{\Sing_\frak p})=
\on{td}^{\vee\!}\omega_{\widetilde \pi}(\varphi^*\on{td}^{\vee\!}{\ca O}_{\Sing^\sim})^{-1}.
%=
%\on{td}^{\vee\!}\omega_{\tilde \pi}+
%(\varphi^*\on{td}^{\vee\!}{\ca O}_{\Sing})^{-1}-\pmb 1.
\end{multline*}%where the last equality holds because $\omega_{\tilde \pi}$ is trivial at the nodes.
The first factor is $$(1+K)=\on{td}^\vee(\omega_{\widetilde\pi})= \kappa^*
\on{td}^\vee(\omega_{\pi}),$$ 
and the second is of the form \begin{equation}\label{eq:mumfordlocal}
\pmb 1+N=(\varphi^*\on{td}^{\vee\!}{\ca O}_{\Sing^\sim})^{-1}
=\pmb 1 + \widetilde i_*P(c_1(N_{\tilde i}), c_2(N_{\tilde i}))
% \frac{1}2 \widetilde i_*\left( \frac{[\frac1{e^{\widetilde\psi}-1}]_{\ge 1} + [\frac1{e^{{\widetilde\psi}'}-1}]_{\ge 1}} {\widetilde\psi+{\widetilde\psi}'} \right),
\end{equation}
for a universal power series $P$ in the 
Chern classes of the normal bundle of 
$\widetilde i$, see \cite[Lem.~5.1]{Mumford1983Towards-an-enum}.

% where $[\ \ ]$ denotes the truncation to the terms of 
% degree $\ge 1$ in the variables $\widetilde\psi$ and $\widetilde\psi'$.
%We get  \begin{equation}\label{eq:untwisted1}
%\exp\left(\frac\mu r\right)\frac{c_1(\omega_{\widetilde \pi})}{e^{c_1(\omega_{\widetilde\pi})} -1 }
%%=\sum_{m\ge 0} \frac1{m!}\ca B_m\left(\frac{c_1\ca L}{r}, c_1\omega\right).\end{equation}
%%We re-write the contribution from the nodes 
%%\begin{equation}\label{eq:untwisted2}
%+r\exp\left(\frac\mu r \right)
%i_*\left( \frac{1}{\psi+\psi'} 
%\left[\frac{1}{e^{\psi/r} -1} \right]_{\deg_\psi\ge 1} \right).\end{equation}

We have reduced the computation to 
pushing forward via $\widetilde\pi$
the product of the four terms $({{\pmb 1}+U})
({{\pmb 1}+E})({{\pmb 1}+K})({{\pmb 1}+N})$, 
where $N$ is supported on the nodes, 
$E$ on the exceptional divisor, $K$ on the canonical 
divisor and ${\pmb 1}+U$ is the pullback of the universal expression 
${\pmb 1}+\ca U=\exp(c_1(\ca L)/r)$  from $C$. 
The products $KN$ vanishes, because $\omega_{\widetilde\pi}$
is trivial along the nodes. 
Therefore the above four-term product
reduces to 
\begin{equation}\label{eq:threeterms}
({{\pmb 1}+U})({\pmb 1}+K) + ({{\pmb 1}+U})\big(({\pmb 1}+E) ({\pmb 1}+N)-{\pmb 1}\big)+
({\pmb 1}+U)KE
\end{equation}
and we can ignore the last summand because
$$\widetilde{\pi}_*(({{\pmb 1}+U})KE)={\pi}_*\kappa_*(({{\pmb 1}+U})KE)=\pi_*((({{\pmb 1}+\ca U})\kappa_*(KE)))$$ 
and 
$\omega_{\widetilde \pi}=\kappa^*\omega_{ \pi}$ implies
$\kappa_*(KE)=0$.
In order to see this, we can expand $E=e^{D/r}-\pmb 1$ as a sum of $r-1$ terms of the form $E_k^aP_k$ with $P_k$ a polynomial in 
$E^a_{h}$, for $h=k, k-1$ or $k+1$. 
Then $E^a_k$ can be replaced within $P_k$ by $\widetilde C_\Delta$ minus the closure of its complement within $\widetilde C_\Delta$. 
Ultimately $KE$ boils down to two types of terms:
$(i)$ a sum of terms of classes pulled back from $C$ to $E^a_k$ and pushed down again (yielding 0), or 
$(ii)$ terms containing $KE^a_kF$ where $F$ is a 
divisor meeting $E_k^a$ at a node where $K$ is trivial. The identity $\kappa_*(KE)=0$ follows.

We now compute the push-forward of the two first terms of 
\ref{eq:threeterms}
by projection along $\kappa$. Since $\kappa_*\pmb 1= \pmb 1$, the first summand $({{\pmb 1}+U})({\pmb 1}+K)$ is
\begin{equation}\label{eq:bulk}\exp\left(\frac1r c_1{\ca L}\right)\frac{c_1(\omega_\pi)}{e^{c_1(\omega_\pi)} -1 }=\sum_{m\ge 0} \frac1{m!}{\ca B}_m\left(\frac{\xi}{r}, c_1\omega\right).\end{equation}
For the second term $({{\pmb 1}+U})\big(({\pmb 1}+E) ({\pmb 1}+N)-{\pmb 1}\big)$, since $\pmb 1 + U$ comes from 
$C$, we first need to determine $\kappa_*(({\pmb 1}+E) ({\pmb 1}+N)-{\pmb 1})$. 
We can  rely 
on \cite{Chiodo2008Towards-an-enum} 
where this push-forward vie $\kappa$ on $C$ (matching \cite[eq. (20)]{Chiodo2008Towards-an-enum}) is explicitly identified 
in the Chow ring  with
\begin{equation*}
%  \kappa_*\left(   \on{ch}\left(-\frac{D_a}r\right)\left(\pmb 1 + 
% \frac12 i_* {\on{td}}^{\vee}{\ca O}_{\on {\SSing}}\right)^{-1}\right)=\\ 
\pmb 1 + \frac{r}2 \sum_{d\ge 1}\frac{B_{d+1}(\frac{a}r)}{(d+1)!}i^a_*\frac{\psi^{d}-(-\psi')^d}{\psi + \psi'},
\end{equation*}
(see \cite[eq. (21)]{Chiodo2008Towards-an-enum};
the computation is carried out in \cite[Step 3, p.1475]{Chiodo2008Towards-an-enum} and the same argument is used in 
%the latter that eq. (21) is written in %the above form in the last equation of %the proof of the main theorem at page %1481, see also
\cite[eq.(10), \S2.4]{Janda2018Double-ramifica}).
By taking into account the product with  
$\pmb 1 + U$, we get 
\begin{equation}\label{eq:theformulawithja}
\frac{r}2 \sum_{a=0}^{r-1}j^{a}_* 
\sum_{\substack{p+q=m+1\\ p\ge 2}}\frac{(\xi_a)^q}{r^q} \frac{B_p(\frac{a}{r})}{p!q!} 
\frac{\psi^{p-1}-(-\psi')^{p-1}}{\psi+\psi'}.
\end{equation}
To obtain \ref{introtheorem:ch_formula}, we break $\SSing_a$ up into a union of irreducible components $\SSing_h$ for each half-edge, and re-write in the notation of tautological classes on $\PIC^{1/r}$ introduced in \ref{sec:intro_RR}, noting that a fibre of $\SSing = \bigsqcup_{a=0}^{r-1} \SSing_a$ is nonempty if and only if the coefficients $a$ form a weighing mod $r$ balancing the multidegree of $\ca L$.

\subsection{Proof of \ref{preintrotheorem:Pixton}}

\ref{preintrotheorem:Pixton} can be deduced from \ref{introtheorem:ch_formula} in exactly the same way as \cite[Proposition 5]{Janda2016Double-ramifica} is deduced from \cite[Theorem 1.1.1]{Chiodo2008Towards-an-enum}. 

We write out some details of the argument, because the tautological classes on $\PIC^{1/r}$ are less standard than those on spaces of roots over $\Mbar_{g, n}$, and because an intermediate computation will be important in the proof of \ref{introtheorem:DR}.

Given positive integers $m$ and $r$, let
\begin{equation}
a_m \coloneqq (-1)^m \hat \pi_* \frac{\mathcal B_{m+1}\left(\frac{c_1\ca L}{r}, c_1\omega\right)}{m(m+1)}, 
\end{equation}
and given also a  codimension-1 decorated graph $\tilde \Gamma = (\Gamma, \delta, w) \in \rgraphs^{1/r}(1)$ corresponding to a stratum $\PIC^{1/r}_{\tilde \Gamma} \to \PIC^{1/r}$, define
% \david{Ale says the $\psi_h$ below does not mean the same as the $\psi_k$ in the previous subsection. But D does not completely understand. Is the issue just that here we did not blow up? But in the first version of the proof of Theorem C I think there are $\psi$ classes at the stacky nodes, and these are basically those? }\ale{I think  it is fine now.}
\begin{equation}\label{eq:b_def}
b_{{\tilde \Gamma},m} = \sum_{\substack{{p+q=m+1}\\{p\ge 2}}}(-1)^{m-1}(m-1)!
 \frac{B_{p}(\frac{w(h)}{r})}{p!q!\abs{\on{Aut}(\tilde\Gamma)}}\xi_e^{q} (\psi_h^{p-1} - (-\psi_{h'})^{p-1})
\end{equation}
in $\Chow(\PIC^{1/r}_{\tilde\Gamma})$, where the unique edge $e$ of $\Gamma$ has half-edges $h$ and $h'$, and $N_{\tilde \Gamma} = -({\psi_h + \psi_{h'}})/{r}$ is the first Chern class of the normal bundle to the codimension-1 stratum $\PIC^{1/r}_{\tilde\Gamma}$. %\david{Is is this or is it minus the pullback of the stratum in the non-rooted space? For now I choose to believe the formula $N_{\tilde \Gamma} = \psi_h + \psi_{h'}$ (where the psi classes are pulled back from the base), so the actual normal bundle in $\PIC^{1/r}_{\tilde\Gamma}$ is $-\frac{N_\Gamma}{r}$? }

%\begin{equation}
%b_{{\tilde \Gamma},m} = N_{\tilde \Gamma} \cdot \sum_{h \in H(\Gamma)} \sum_{\substack{{p+q=m+1}\\{p\ge 2}}}(-1)^m(m-1)!
% \frac{B_{p}(\frac{w(h)}{r})}{p!q!r^q\abs{\on{Aut}(\tilde\Gamma)}}(k^h)_*(i^{h})^*\left((c_1\ca L)^{q} \frac{\psi^{p-1} - (-\psi')^{p-1}}{\psi + \psi'}\right). 
%\end{equation}
%where we write $N_{\tilde \Gamma} = \psi_h + \psi_{h'}$ for the first Chern class of the normal bundle to the boundary divisor $\PIC^{1/r}_{\tilde\Gamma}$ corresponding to $\tilde \Gamma$, and  $k^h\colon \SSing_h \to \PIC^{1/r}_{\tilde \Gamma}$ is the factorisation of $j^h$. 

\begin{lemma}\label{lem:total_chern_before_push}
The total Chern class  $c(-R\pi_*\ca L^{1/r})$ equals %\david{Could cancel the $r$ in the denominator with the $r^E$ in the numerator; depends whether we want to imitate JPPZ or have the more natural expression. }\ale{The proofs are very readable in this way, and this is a technical statement, no need to simplify the $r^{\abs{E(\tilde \Gamma)}}$.} 
\begin{equation*}
   \left(\exp{\sum_{m \ge 1} a_m}\right)\sum_{\tilde \Gamma \in \rgraphs^{1/r}} \frac{r^{\abs{E(\tilde \Gamma)}}}{\abs{\on{Aut}(\tilde  \Gamma)}} {j_{\tilde \Gamma}}_* \left[ \prod_{{e} \in E(\tilde \Gamma)} \frac{1-\exp{\sum_{m \ge 1} b_{{e},m}}}{-rN_{e}} \right],
\end{equation*}
where $j_{\tilde \Gamma}$ is the finite map from $\PIC^{1/r}_{\tilde\Gamma}$ to $\PIC^{1/r}$. The product is taken over edges of $\tilde \Gamma$, and $b_{e,m}$ denotes the class from \ref{eq:b_def} applied to the graph in $\rgraphs^{1/r}(1)$ obtained by contracting all edges of $\tilde \Gamma$ except for $e$. 
%The product is interpreted as a class on $\PIC^{1/r}_{\tilde\Gamma}$ following \david{Can refer to Section 0.3.3 of BHPSS, though we may want to write more in our intro (note BHPSS doesn't do roots, but that't not an important distinction). Anyway I need to think about this, and work out what to write in the introduction, as we probably have to explain this at least for $r=1$. }
\end{lemma}
\begin{proof}The proof is based on that of \cite[Proposition 4]{Janda2016Double-ramifica}. We re-write \ref{introtheorem:ch_formula} as
\begin{equation}
(-1)^m(m-1)!  \on{ch}_{m}(R\pi_*\ca L^{1/r}) = a_m + \sum_{\tilde \Gamma \in \rgraphs^{1/r}(1)}  (j_{\tilde \Gamma})_* \left(\frac{b_{\tilde \Gamma, m}}{N_{\tilde \Gamma}}\right)
\end{equation}
% \david{checked powers of $r$ and 2. Ale, you had some comment here but not sure if it was something we need to check/change, or just a note to yourself. }
and apply the standard formula
\begin{equation}\label{eq:chern_char_class}
c(-E^\bullet) = \exp\left( \sum_{m \ge 1} (-1)^m (m-1)!\on{ch}_m(E^\bullet)\right), \end{equation}
along with the observation that, for $\tilde \Gamma \in \rgraphs^{1/r}(1)$ and classes $Z_1, \dots, Z_n\in \Chow(\PIC^{1/r}_{\tilde \Gamma})$, we have 
\begin{equation}
\prod_{i=1}^n (j_{\tilde \Gamma})_*(Z_i) = \abs{\on{Aut}(\tilde \Gamma)}^{n-1}(j_{\tilde \Gamma})_*\left(N_{\tilde \Gamma}^{n-1}\prod_{i=1}^n Z_i\right)
\end{equation}
whenever $n \ge 1$. 
\end{proof}

Recall that $\epsilon\colon \PIC^{1/r} \to \PIC$ is the $r$th power map, finite flat of degree $r^{2g-1}$. 
\begin{lemma}\label{lem:pushforward_along_epsilon}
The direct image $\epsilon_* c(-R\pi_*\ca L^{1/r})$
equals%\david{Subscript should be $\tilde \Gamma \in G$ not $\in G(1)$? }\ale{I agree, but my opinion here is worth much less than yours here.
%I can see that writing $G(1)$ (one edge) and then taking the product on the edges makes little sense.}
\begin{equation}
 \left( \exp{\sum_{m \ge 1} a_m}\right)\sum_{\tilde \Gamma \in \rgraphs}  \frac{r^{2g -1 - h^1(\Gamma)}}{\abs{\on{Aut}( \tilde \Gamma)}}  (j_{\tilde \Gamma})_* \left[ \prod_{e\in E(\tilde \Gamma)} \frac{1-\exp{\sum_{m \ge 1} b_{m,e}}}{N_e} \right].
\end{equation}
\end{lemma}
%\david{On the RHS, the $a_{m}$ and $b_{m, e}$ are defined on $\PIC$ and not $\PIC^{1/r}$, I think? This is kind of OK, but the notation sa written is confusing! I think we probably do need an $r$ in the $a_m$ and $b_{m,e}$ notation (or maybe in the $\rgraphs$ notation), else it is confusing. }
\begin{proof}
    All the terms of \ref{lem:total_chern_before_push} come by pullback from the case $r=1$ (indeed, the notation looks almost the same, because to avoid clutter we do not decorate the classes with pullbacks everywhere). The only thing to check is the leading power of $r$. Here we argue exactly as in \cite[Corollary 4]{Janda2016Double-ramifica}. %\Dcomment{TODO eventually: copy the argument here, to make the paper self-contained. But not urgent. }
\end{proof}

We define a class
\begin{multline}\label{eq:Pixton_strata_version}
P(r) \coloneqq \left(\exp{  -\frac{1}{2}\hat \pi_* (c_1\ca L)^2} \right)\cdot\\
 \sum_{\tilde \Gamma \in \rgraphs(1)}  \frac{r^{- h^1(\Gamma)}}{\abs{\on{Aut}(\tilde \Gamma)}} \sum_{w\in W_r(\tilde \Gamma)} {j_{\tilde\Gamma}}_* \left[ \prod_{e} \frac{1-\exp{(\frac{w(e)^2}{2} N_e})}{N_e} \right]
\end{multline}
in the Chow ring of $\PIC$, and we define $P_d(r)$ to be the codimension-$d$ part. 

\begin{proposition}[{\Cref{preintrotheorem:Pixton}}]\label{lem:c_d_pixton_equal} We restrict now to the total-degree-zero part of the Picard stack. 
Fix a positive integer $d$. Then the constant terms of the polynomials in $r$ given by 
    \begin{equation}
r^{2d+1-2g}    \epsilon_* c_d(-R\pi_*\ca L^{1/r})  \text{ and } P_d(r)
    \end{equation}
are equal in the Chow ring of $\PIC^{\mathrm{tot}0}$.
\end{proposition}
\begin{proof}
The fact that these are polynomial in $r$ uses \cite[Proposition $3''$]{Janda2016Double-ramifica} to show that the positive contribution to $r$ from summing over the weighting cancels with the negative contribution $r^{-h_1(\Gamma)}$. 
To equate the constant terms, the key observation is that any terms in the expression from \ref{lem:pushforward_along_epsilon} corresponding to $m>1$ can be discarded as they cannot contribute to the constant term (this is the same argument as in \cite[Corollary 5]{Janda2016Double-ramifica}). The formula is then immediate from the equality $\ca B_2(x,y) = x^2 - xy +y^2/6$. 
\end{proof}

\subsection{The class over  moduli  of roots on $r$-stable curves}\label{sec:formula_on_Mbar}

In this section we see how \ref{introtheorem:ch_formula}  (in the shape of \ref{ch_formula_withBH}) specialises  over the moduli stack of $r$th roots of 
a 
line bundle
$\ca L$ on the universal stable curve over $\Mbar_{g,n}$. For $r=1$ the formula recovers the equation 
given by Pagani, Ricolfi, and van Zelm \cite{Pagani2020Pullbacks-of-un}, see \ref{rem:compatibility}.

% For 
% any set $I$,
% the notation
% $\Mbar_{g,I}$ stands for the moduli stack 
% of stable curves of genus $g$ with markings labelled
% by $I$.

\begin{notation}[the divisors $S_i$]
Within the universal stable curve over $\Mbar_{g,n}$, for $i=1,\dots, n$, the divisors $S_i$ are the images of the sections specifying the 
$i$th marking. 
\end{notation}

\begin{notation}[large pairs,  complementary pairs, and the relations $\preceq$ and $\prec$] \label{notn:small}
We write $\delta_{g,n}$ for the set of pairs 
$(h,N)$ with $h\in\{0,\dots, g\}$ and $N\subseteq \{1, \dots, n\}$. The \emph{pair} $(h,N)$ is \emph{large} if, for $n>0$, $N$ contains the index $1$, and,  
for $n=0$, we have $h>g/2$. We write $\delta^+_{g,n}$ for the set of large pairs. 
The \emph{complementary} pair to $(h,N)$
is given by $-(h,N)\coloneqq(h'=g-h, N'=\{1, \dots, n\}\setminus N)$.
Finally, we write $(h,N)\preceq(k,M)$ if the two conditions $h\le k$
and $N\subseteq M$ are satisfied.
We write  $(h,N)\prec (k,M)$ if 
at least one of the two conditions  is strict (\emph{i.e.} if 
$(h,N)\prec (k,M)$ and $(h,N)\neq (k,M)$).
Notice there is at most one large pair among the complementary pairs 
$(h,N)$ and $-(h,N)$, and the only case where 
neither of them is large is  $(h,N)=(g/2,\varnothing)$ with $n=0$.
\end{notation}

% We now introduce the divisors $D_h^N$
% for any  pair $(h,N)$.
\begin{notation}[the divisors $D_h^N$]\label{notn:vertdivisors}
Let $(h, N)\in \delta_{g,n}$. Identifying the universal curve $\pi\colon \ca C\to \Mbar_{g,n}$ with $\Mbar_{g, n+1}$, we define $D_h^N$ to be the boundary divisor in $\ca C$ corresponding to the one-edge genus-$g$, $n+1$-marked graph with one vertex decorated 
by $(h,N)$ (automatically the remaining vertex is decorated by $h'$ and $N'\sqcup \{n+1\}$, where $(h', N')$ is the complementary pair). 
We refer to $D_h^N \cap D_{h'}^{N'}$ (the node where 
$D_{h}^N$ meets the rest of the fibre) as a separating
node of type $(h,N)$. 
\end{notation}
For the remainder of this section $r$ denotes a positive integer. We fix a line bundle $\ca L$ on the universal curve over $\Mbar_{g,n}$. 
There exist unique integers $a_1, \dots, a_n, s$ and $l_h^N$ for any large pair $(h,N)$ such that 
\begin{equation}\label{eq:formL}
  \ca L
\cong (\omega_{\on{log},\pi})^{\otimes s}\Big(-\sum_{i=1}^n a_iS_i +\sum_{(h,N)\in \delta^+_{g,n}} l_{h}^N D_{h}^N\Big)
\end{equation}
up to pullback from the moduli space
$\Mbar_{g,n}$. This is 
a version of the Franchetta conjecture proven in \cite{harer1983second}, 
 %The second homology group of the mapping class group of an orientable surface, Invent.
%Math. 72 (1983), no. 2, 221–239. MR 700769)}. 
\cite{arbarello1987picard}, and
%Enrico Arbarello and Maurizio Cornalba, The Picard groups of the moduli spaces of curves, Topology
%26 (1987), no. 2, 153–171. MR 895568
\cite{mestrano1987conjecture}.
% and for a unique choice of the integers
% $l^V_{v}$, $a_i$ and $s$ (the second sum is taken only
% over large pairs). 
%We write $\Mbar_{g,n}^{\ca L/r}$ for the 
%moduli stack of $r$th roots of $\ca L$ on 
%$r$-stable twisted curves. 
%In order to apply  \ref{introtheorem:ch_formula} we provide 
%explicit expressions for the morphisms $j_a$, 
%the coefficients $a/r$, and the classes $c_1(
%\ca L)$.
We assume that the total degree of $\ca L$ is divisible by $r$; in other words, the coefficients
$a_i$ and $s$ satisfy that $(2g-2+n)s -\sum_i a_i\in r\mathbb Z$.
We also notice that $l_{N}^h$ are only determined when $(h,N)$ is large; 
we set them to $0$ otherwise.

%We  decompose the double cover $\SSing$  of the 
%singular locus $\Sing$ of the universal curve over
%$\Mbar{}_{g,n}^{\ca L/r}$ 
%as follows. 
\begin{notation}[the singular locus]\label{notn:singlocusdecomptype}
The nodes are either separating (s) 
or nonseparating (ns); we have 
$\SSing= \SSing_{\on{ns}} \sqcup \SSing_{\on{s}}$.
The stack $\SSing_{\on{s}}$ can be decomposed according
to the type (see \ref{notn:vertdivisors}) as a union of 
open and closed substacks $\SSing_{k,M}$ for $(k,M)\in \delta_{g,n}$
corresponding to the branch over the node of type $(k,M)$ lying on $D_{k'}^{M'}$ for 
$(k', M')=-(k,M)$. The term $\SSing_{\on{ns}}$ can be decomposed as 
$\bigsqcup_{a=0}^{r-1}\SSing_{\on{ns}}^a$ according to the multiplicity of $\ca L$.
On the other hand, for each type $(k,M)$ the multiplicity 
of $\ca L$ on the branch belonging to $D_{k'}^{M'}$
equals 
 $a_{k,M}\in \{0,\dots, r-1\}$ defined by
$$a_{k,M}\equiv (2k-1)s -  \sum_{M} (a_i-s) +\begin{cases} l_{k}^{M} & \text{if $(k,M)$ is large},\\
-l_{k'}^{M'} & \text{if $(k',M')$ is large},\\
0 & \text{else}\end{cases}$$ 
 modulo $r$.
%  of the degree 
%of $\ca L$ on the curves of $D_{k'}^{M'}$, and coincides
%with the multiplicity of $\ca L$
%(if the action 
%on the chosen branch is given by multiplication by $\zeta\in \pmb\mu_r$, the action on the fibre of $%\ca L$ is 
%given by multiplication by $\zeta^a$ with $a=a_{k,M}$). 
%(it yields opposite values for 
%complementary pairs $(w,W)$ and 
%$(w', W')$ because $r$ divides $(2g-2+n)s-\sum_i a_i$). 
%In other words, the stack $\SSing_{w,W}$ satisfies $\SSing_{w,W}\subseteq 
%\SSing_{a} \ \text{for $a=a_{w,W}$.}$
%This allows us to refine the decomposition 
%\ref{eq:SingDecomp} as follows.
We have 
\begin{equation}\label{eq:SingDecomp}
\SSing= \bigsqcup_{a=0}^{r-1}\SSing_{\on{ns}}^a \sqcup \bigsqcup_{(k,M)\in \delta_{g,n}} \SSing_{k,M}
\end{equation}
alongside with maps $\bigsqcup_{a=0}^{r-1}i_{a,\on{ns}}\sqcup \bigsqcup
i_{k,M}$ to the universal curve $\ca C$ and 
$\bigsqcup_{a=0}^{r-1}j_{a,\on{ns}}\sqcup \bigsqcup
j_{k,M}$ to the moduli stack $\Mbar_{g,n}^{\ca L/r}$. 
\end{notation}
\begin{remark}[intersecting the singular locus and the vertical 
divisors $D_h^N$]
The vertical divisor $D_h^N$ meets the image of $i_{k,M}$ if and only if 
$(h,N)\preceq(k,M)$ or $(h,N)\preceq(k',M')$ which we can write concisely as $(h,N)\preceq\pm(k,M)$ using \ref{notn:small}. 
If $(h,N)$ is large, a vertical divisor $D_{h,N}$ meets a singular locus $\on{im}(i_{k,M})$
only if $(k,M)$ is large 
or $-(k,M)$ is large and the two conditions $(h,N)\preceq(k,M)$ and $(h,N)\preceq-(k,M)$ 
 cannot
be satisfied 
simultaneously. For the same reasons two  vertical divisors $D_{h_1,N_1}$ 
meet each other $D_{h_2,N_2}$ if and only if $(h_1,N_1)\preceq (h_2,N_2)$ or 
$(h_1,N_1)\succeq (h_2,N_2)$ and the relation is strict if they are distinct. 
 \end{remark}

If the  
relation is strict, \emph{i.e.} $(h,N)\prec (k,M)$,
the divisor $(i_{k,M})^*D_{h}^N$ 
is the pullback of $D_h^N$ from 
$\Mbar_{g,n}$ with a factor $r$ due to the ramification of
the stack of $r$-stable curves over the stack of stable curves $\Mbar_{g,n}$. Instead, if $(h,N)=(k,M)$ we get 
the first Chern class of the tangent line $-\psi/r$, 
 multiplied again by $r$. We get the following lemma, 
 where we abused  notation and wrote $D_{h}^N$ for 
the pullback of $D_h^N$ from $\Mbar_{g,n}$ to $\SSing_{k,M}$.
We write 
$l_k^M\psi + l_{k'}^{M'}\psi'$ which equals the first
or the second summand
depending on whether we have 
$(h,N)\prec(k,M)$ or $(h,N)\prec-(k, M)=(k',M')$ ($l_k^M$ is zero if $(k,M)$ is not large). 
The statement also contains the computation of 
$(i^{a}_{\on{ns}})^*c_1\ca L$, which follows by the same 
argument as above. 
\begin{lemma}\label{lem:seprest}
For any pair $(k,M)\in \delta_{g,n}$, on $\SSing_{k,M}$ we have 
\begin{equation}\label{eq:seprest}(i_{k,M})^*c_1\ca L= -{l_{k}^M}\psi-l_{k'}^{M'}\psi'+ r\sum\limits_{(h,N)\prec\pm(k,M)}l_{h}^N D_{h}^{N}.
\end{equation}
Furthermore, for any $a\in \{0,\dots,r-1\}$, we have 
\begin{equation}\label{eq:nsrest}
(i^{a}_{\on{ns}})^*c_1\ca L=r\sum_{(h,N)\in \delta_{g,n}} l^N_{h} 
D^{N}_{h}. 
\end{equation}\qed
 \end{lemma}

The boundary contributions from the second summand of \ref{introtheorem:ch_formula}
can now be written\footnote{Note that, as in \ref{eq:theformulawithja} and unlike \ref{introtheorem:ch_formula}, we consider the pushforward via $j^{a}$ instead of  $j_{\tilde \Gamma}$. This 
requires a factor $1/2$ which we omit because, 
with respect to \ref{eq:theformulawithja},
we replace the $\sigma$-involution invariant 
expression $\frac12B_p(a/r)(\psi^{p-1}-(-\psi')^{p-1})/(\psi+\psi')$ (for $p\ge 2$)
by the first summand $B_p(a/r)\psi^{p-1}/(\psi+\psi')$ which is not $\sigma$-invariant 
but matches \ref{eq:theformulawithja} globally after summing over all pushforwards $\sum_a j^{a}_*(\dots)$ as a consequence of  
$B_p(a/r)=(-1)^pB_p(1-a/r)$ for $a\neq 0$ 
and $B_p(0)=(-1)^pB_p(0)$ for $p\ge 2$. This observation 
allows us to conveniently rewrite the formula in terms of homogenized Bernoulli polynomials.}  as 
%\begin{theorem}\label{thm:ch_formula_correction} 
% We have 
\begin{multline}\label{ch_formula_withBH} 
\on{ch}_{m}(R\pi_*\ca L^{1/r})=\pi_* \frac{\mathcal B_{m+1}\left(\frac{\xi}{r}, c_1\omega\right)}{(m+1)!}+\\
{r} \sum_{0\le a<r} j^{a}_*
\frac{[\mathcal B_{m+1}(\frac1r({a}\psi_h+\xi_a), \psi_h )]_{\deg_{\psi_h}\ge 1}}{(m+1)!\psi_h(\psi_{h}+\psi_{h'})},
\end{multline}
%\end{theorem}
where the term $[\mathcal B_{m+1}(\frac1r({a}\psi+\xi_a),\psi)]_{\deg_\psi>1}$
is the homogenized  polynomial without\footnote{Actually, the terms in degree $1$ in the variable $\psi$ corresponding
to $a>0$ cancel out with the terms in degree $1$ 
for $a'=r-a$ in the class $\psi$. Therefore, the numerator of the second 
summand can be written as
$
\mathcal B_m\left(\frac{a}{r}\psi+\xi_e,\psi\right)-\left(\xi_e\right)^m+\frac{\delta_{a,0}}2\psi \left({\xi_e}\right)^{m-1}.$} all terms of degree $0$ and $1$ in the variable $\psi$.

By \ref{lem:seprest} we  
express $\xi_a$ in terms of the divisors $D_{h,N}$ and the psi-classes, which we handle via   
$\mathcal B_m(x+y,z)=
\mathcal B_{m}(y,z)+\sum_{j=1}^m \binom{m}{j} x^j
\mathcal B_{m-j}(y,z)$. In 
this way the usual 
truncation $[\cdots]:=[\cdots]_{\deg_\psi>1}$ is compatible with the extra psi-classes and we get the
following expression for the second summand of $\on{ch}_m(R\pi_*\ca L)$ in
\ref{ch_formula_withBH}.
\begin{multline*}
{r} \sum_{0\le a<r} j^a_{\on{ns}, *}
\frac{[\mathcal B_{m+1}(\frac{a}r\psi+ \sum_{h,N} l_h^ND_h^N,\psi)]}{(m+1)!\psi(\psi+\psi')}+\\
{r}\sum_{(k,M)\in \delta_{g,n}} j^{k,M}_{*}\Bigg(
\frac{[\mathcal B_{m+1}(\frac{a_{k}^M}r\psi+ \sum_{(h,N)\prec\pm(k,M)} l_h^ND_h^N,\psi)]}{(m+1)!\psi(\psi+\psi')}+
 \\
 \sum_{\stackrel{u+v=m+1}{s>0}}\left(-\frac{l_k^M\psi+l_{k'}^{M'}\psi'}{u!r}\right)^{\! u}
 \frac{[\mathcal B_{v}(\frac{a_{k}^M}r\psi+ \sum_{(h,N)\prec\pm(k,M)} l_h^ND_h^N,\psi)]}{v!\psi(\psi+\psi')}\Bigg).
\end{multline*}

    Since two distinct vertical divisors  $D_{h_1}^{N_1}$ and $D_{h_2}^{N_2}$ 
intersect each other if and only if $(h_1,N_2)\prec(h_2,N_2)$ 
or $(h_2,N_2)\prec(h_1,N_1)$, we can further simplify
 the
above expression
by 
taking the sum over the set of strictly \emph{increasing}  sequences of  large pairs
nested as follows
$$(\pmb h, \pmb N)_c
=(h_1,N_1)\prec\dots\prec (h_c,N_c),$$
for $c\ge 0$. For each of these 
configurations we truncate further 
the Bernoulli polynomials so that the 
intersections are supported on the strata
of the boundary of $\Mbar{}_{g,n}^{{\ca L}/r}$ underlying $\bigcap_i D_{h_i}^{N_i}$ for  $i=1,\dots,c$.

We write 
$\Sing^{(\pmb h, \pmb N)_c}$ 
or simply $\Sing^{\pmb h, \pmb N}$ for 
the singular locus (of any type, $\on{ns}$ or $(k,M)$) 
occurring within a curve which possesses a 
sequence of distinct separating 
nodes $n_1,\dots, n_c$ each one 
 yielding after normalisation a disconnected curve   
 with two components, one containing the node $n$
and the other one  having genus $h_i$ and  bearing the marking 
 set $N_i$ (if $c=0$ this is simply the definition of the singular locus $\Sing$ of $\ca C$).
As in \ref{eq:SingDecomp}, $\SSing^{(\pmb h, \pmb N)_c}$ denotes
the double cover classifying the two branches of $n$;
it coincides with $\SSing$ when $(\pmb h, \pmb N)$ 
has no entries (\emph{i.e.} $c=0$) and --- in complete analogy
with $\SSing$ --- it is equipped with the psi-classes
$\psi$ and $\psi'$
and it splits as 
$$\SSing^{\pmb h, \pmb N}= 
\bigsqcup_{a=0}^{r-1}\SSing_{a,\on{ns}}^{\pmb h, \pmb N}\sqcup 
\bigsqcup_{(k,M)\in \delta_{g,n}}
\SSing_{k,M}^{\pmb h, \pmb N}$$
mapping  
 via 
 $\bigsqcup_{a=0}^{r-1}i^{\pmb h, \pmb N}_{a,\on{ns}}\sqcup \bigsqcup
i_{k,M}^{\pmb h, \pmb N}$ to the universal curve $\ca C$ and 
$\bigsqcup_{a=0}^{r-1}j^{\pmb h, \pmb N}_{a,\on{ns}}\sqcup \bigsqcup
j_{k,M}^{\pmb h, \pmb N}$ to the moduli stack $\Mbar{}_{g,n}^{\ca L/r}$.
Furthermore
it maps to the divisor $D_{h_i}^{N_i}$ for all $i$ and the first Chern class of the 
pullback of ${\ca O}(D_{h_i}^{N_i})$ on it 
is the normal class $\nu_i$ which 
can be expressed as the opposite of 
the sum of the psi-classes associated to the two branches of the nodes $n_i$ of type $(h_i,N_i)$. 
\begin{notation}
[the truncated Bernoulli polynomial] 
Below we systematically evaluate the homogenized Bernoulli polynomials
$\mathcal B_n(x,y)$ 
on the variables $x=a\psi+ \sum_{i} b_i\nu_i$ and $y=\psi$ for some rational coefficients 
$a$ and $b_i$. 
Let  ${\mathcal{TB}}_n(a\psi+ \sum_{i} b_i\nu_i, \psi)$ 
the homogenized Bernoulli polynomial $\mathcal B_n$ 
truncated to its terms of positive degrees in each of the variables $\nu_i$ 
and to its terms of degree higher than $1$ in the variable $\psi$. 
 \end{notation}Then, we can rewrite \ref{introtheorem:ch_formula} 
over $\Mbar^{\ca L/r}_{g,n}$ in 
the following way.
\begin{corollary}
Over $\Mbar^{\ca L/r}_{g,n}$, we have  
\begin{multline*}%\label{ch_formula_withBH} 
\on{ch}_{m}(\pi_!\ca L^{1/r})=\pi_* \frac{\mathcal B_{m+1}\left(\frac{c_1\ca L}{r}, c_1\omega\right)}{(m+1)!}\\+
{r}\sum_{\stackrel{(\pmb h, \pmb N)_c}{c\ge0}}\Bigg(\sum_{a=0}^{r-1}
(j_{a,\on{ns}}^{\pmb h, \pmb N})_*
\frac{{\mathcal{TB}}_{m+1}(\frac{a}r\psi+ \sum_{i=1}^c  l_{h_i}^{N_i}\nu_i,\psi)}
 {(m+1)!\psi(\psi+\psi')\prod_{i=1}^c \nu_i}
\\+
\sum_{(\pmb h, \pmb N)_c\prec \pm(k,M)} (j_{k,M}^{\pmb h, \pmb N})_*
 \frac{{\mathcal{TB}}_{m+1}(\frac{a_k^M}r\psi+ \sum_{i=1}^c  l_{h_i}^{N_i}\nu_i,\psi)}
 {(m+1)!\psi(\psi+\psi')\prod_{i=1}^c \nu_i}
 \\+ (j_{k,M}^{\pmb h, \pmb N})_*
 \sum_{\stackrel{s+t=m+1}{u>0}} \left(-\frac{l_k^M\psi+l_{k'}^{M'}\psi'}{u!r}\right)^{\!\!\!u}
  \frac{{\mathcal{TB}}_v(\frac{a_k^M}r\psi+ \sum_{i=1}^{c}  l_{h_i}^{N_i}\nu_i,\psi)}
  {v!\psi(\psi+\psi')\prod_{i=1}^c \nu_i}\Bigg).
\end{multline*}\qed
\end{corollary}

\begin{remark}\label{rem:compatibility} The three boundary terms  within brackets
parametrised by the strictly increasing 
sequences $(\pmb h,\pmb N)_c$ correspond to the 
terms 
$\widetilde{\pmb Y}$,  
$\widetilde{\pmb X}$ 
and $\pmb X$
in \cite[Theorem 1.1]{Pagani2020Pullbacks-of-un}
after summation over the parameters 
$a, b$ and $e$ appearing there.
% and repeated usage of $
% B_n^H(x+y,z)=\sum_{j=0}^n \binom{n}{j} x^j
% B_{n-j}(y,z)$.
\end{remark}

\bibliographystyle{alpha} %amsplain}
\bibliography{prebib.bib}

\end{document}